\documentclass[psamsfonts, reqno, 11pt]{amsart}
\usepackage{fullpage,amssymb,stmaryrd}
\usepackage{hyperref}
\usepackage[utf8]{inputenc}
\usepackage{amsfonts}
\usepackage[hyphenbreaks]{breakurl}
\hypersetup{
	pdftitle={On the Beilinson--Bloch conjecture over function fields},
	pdfsubject={Mathematics, Algebraic Geometry, Arithmetic Geometry},
	pdfauthor={Matt Broe},
	pdfkeywords={},
	breaklinks=true
}
\usepackage{xurl}
\usepackage{amsmath}
\usepackage{xcolor}
\usepackage{xparse}
\usepackage{amsthm}
\usepackage{pgfplots}
\usepackage{mathrsfs}
\usepackage[backend=biber,style=alphabetic,giveninits=true,maxnames=50,]{biblatex}
\usepackage{tikz-cd}
\usepackage[shortlabels]{enumitem}
\usepackage[autostyle, english = american]{csquotes}
\DeclareUnicodeCharacter{0335}{}

\DeclareFieldFormat{doi}{%
  \mkbibacro{DOI}\addcolon\space
  \ifhyperref
    {\href{https://doi.org/#1}{\nolinkurl{#1}}}
    {\nolinkurl{#1}}}

\newtheoremstyle{style}
{} 
{} 
{\itshape} 
{} 
{\bfseries} 
{.} 
{.5em} 
{} 

\theoremstyle{style}
\newcommand{\comment}[1]{}
\newtheorem{thm}{Theorem}[section]
\newtheorem{cor}[thm]{Corollary}
\newtheorem{prop}[thm]{Proposition}
\newtheorem{lem}[thm]{Lemma}
\newtheorem{conj}[thm]{Conjecture}

\newtheorem*{thm*}{Theorem}
\newtheorem*{cor*}{Corollary}
\newtheorem*{prop*}{Proposition}
\newtheorem*{lem*}{Lemma}
\newtheorem*{conj*}{Conjecture}
\newtheorem*{quest*}{Question}
\newtheorem*{claim*}{Claim}
\newtheorem*{ppty*}{Property}

\theoremstyle{definition}
\newtheorem{defn}[thm]{Definition}

\theoremstyle{remark}
\newtheorem{rem}[thm]{Remark}

\DeclareMathOperator{\Spec}{Spec}
\DeclareMathOperator{\Gal}{Gal}

\DeclareMathOperator{\Tr}{Tr}

\DeclareMathOperator{\Hom}{Hom}

\DeclareMathOperator{\End}{End}
\DeclareMathOperator{\Fil}{Fil}

\DeclareMathOperator{\id}{id}

\DeclareMathOperator{\Image}{Im}

\DeclareMathOperator{\CH}{CH}

\DeclareMathOperator{\Pic}{Pic}

\DeclareMathOperator{\Griff}{Griff}
\DeclareMathOperator{\Alb}{Alb}
\DeclareMathOperator*{\ord}{ord}

\DeclareMathOperator{\NS}{NS}
\DeclareMathOperator{\Frob}{Frob}

\DeclareMathOperator{\character}{char}

\DeclareMathOperator{\GrpSch}{GrpSch}
\DeclareMathOperator{\rat}{rat}
\DeclareMathOperator{\alg}{alg}
\DeclareMathOperator{\eff}{eff}
\DeclareMathOperator{\num}{num}
\DeclareMathOperator{\cl}{cl}
\DeclareMathOperator{\Jac}{Jac}

\DeclareMathOperator{\red}{red}

\DeclareMathOperator{\op}{op}

\DeclareMathOperator{\gr}{gr}

\DeclareMathOperator*{\colim}{colim}

\DeclareMathOperator{\cont}{cont}

\newcommand{\Z}{{\mathbb{Z}}}

\newcommand{\C}{{\mathbb{C}}}
\newcommand{\F}{{\mathbb{F}}}

\newcommand{\Q}{{\mathbb{Q}}}

\newcommand{\mf}[1]{\mathfrak{#1}}
\newcommand{\mc}[1]{\mathcal{#1}}

\newcommand{\ol}[1]{\overline{#1}}
\newcommand{\ul}[1]{\underline{#1}}

\numberwithin{equation}{section}

\title{On the Beilinson--Bloch conjecture over function fields}

\author{Matt Broe}

\begin{document}
	\linespread{1.3}\selectfont
	\begin{abstract}
		Let $k$ be a field and $X$ a smooth projective variety over $k$. When $k$ is a number field, the Beilinson--Bloch conjecture relates the ranks of the Chow groups of $X$ to the order of vanishing of certain $L$-functions. We consider the same conjecture when $k$ is a global function field, and give a criterion for the conjecture to hold for $X$, extending an earlier result of Jannsen. As a first application, we provide a new proof of a theorem of Geisser connecting the Tate conjecture over finite fields and the Birch and Swinnerton-Dyer conjecture over function fields. 
		\par
		We then prove the Tate conjecture for a product of a smooth projective curve with a power of a CM elliptic curve over any finitely generated field, and thus deduce special cases of the Beilinson--Bloch conjecture. Via related arguments, we obtain a conditional answer to a question of Moonen on the Chow groups of powers of ordinary CM elliptic curves over arbitrary fields.
	\end{abstract}
	
	\maketitle
	
	\tableofcontents

    \section{Introduction}
    \subsection{Conjectures on algebraic cycles}
    Let $k$ be a field and $X$ a smooth variety over $k$. The Chow group of codimension-$i$ algebraic cycles with rational coefficients, denoted by $\CH^i(X)$, is an important and generally mysterious geometric invariant of $X$. Its structure is the subject of numerous outstanding conjectures. When $k$ is a finitely generated field, the Tate conjecture \cite[Conjecture 1]{Tate1965}, as generalized by Jannsen \cite[Conjecture 7.3]{Jannsen1990}, predicts that the $\ell$-adic cycle class map
    \begin{gather*}
        \cl: \CH^i(X) \otimes_\Q \Q_\ell \to H^{2i}(\ol{X}, \Q_\ell(i))
    \end{gather*}
    surjects onto the Galois-invariant subspace of $H^{2i}(\ol{X}, \Q_\ell(i))$. Here $\ell$ is a fixed prime invertible in $k$. When $k$ is finite, a conjecture of Beilinson (also generalized by Jannsen) further predicts that $\cl$ is injective (\cite[page 3]{Beilinson1987}, \cite[Conjecture 12.4 a)]{Jannsen1990}). See \cite{Totaro2017} and \cite[Definition 55]{Kahn2005} for some background and known results on these two conjectures. 
    \par
    Over an infinite field, the cycle class map may have a kernel. We thus let 
    \begin{gather*}
        \CH^i(X)_0 = \ker(\CH^i(X) \to H^{2i}(\ol{X}, \Q_\ell(i)))
    \end{gather*}
    denote the homologically trivial subgroup of $\CH^i(X)$. When $k$ is a global field, and $X$ is smooth and projective over $k$, the Beilinson--Bloch conjecture predicts that
    \begin{gather}\label{eqBBIntro}
        \dim_\Q \CH^i(X)_0 = \ord\limits_{s=i} L(H^{2i-1}(\ol{X}, \Q_\ell), s).
    \end{gather}
    On the right is the order of vanishing at $s=i$ of the Hasse--Weil $L$-function. Note that when $\character(k) = 0$, the well-definedness of the $L$-function at $s=i$ is itself a major outstanding conjecture. The Beilinson--Bloch conjecture was formulated independently by Beilinson \cite[Conjecture 5.0]{Beilinson1987} and Bloch \cite[page 94]{Bloch1984} over number fields (and both attributed it to Swinnerton-Dyer, though to our knowledge he never published it). When $X$ is an abelian variety and $i = 1$, the identity \eqref{eqBBIntro} is equivalent to the rank part of the Birch and Swinnerton-Dyer conjecture for $X$ (hereafter referred to simply as the Birch and Swinnerton-Dyer conjecture). See \cite{Scholbach2017} for a more recent refinement of the Beilinson--Bloch conjecture over number fields. 
    \par
    The Beilinson--Bloch conjecture, as formulated here, is related to other conjectures of Beilinson, Bloch, and Murre concerning filtrations on Chow groups of varieties over arbitrary fields. See \cite{Jannsen1994} for an overview of the latter conjectures.
    \par
    Our results include a sufficient condition for the Beilinson--Bloch conjecture over a global field of positive characteristic. Let $K$ be the function field of a smooth, projective, geometrically integral curve $C$ over $\F_q$, let $G_{\F_q}$ be the absolute Galois group of $\F_q$, let $\Frob_q \in G_{\F_q}$ denote the geometric Frobenius, and let $i \in \Z$.
	\begin{thm}[Corollary \ref{corBBCriterion}]\label{thmBBCriterionIntro}
		Let $f: X \to C$ be a flat, finite-type, separated morphism, with generic fiber $X_K$. If there exists a dense affine open subscheme $U \subset C$ such that
        \begin{enumerate}[(i)]
            \item the base change $f_U: X_U \to U$ of $f$ to $U$ is smooth and projective,
            \item the cycle class map
            \begin{gather*}
                \CH^i(X_U) \otimes_\Q \Q_\ell \to H^{2i}(\ol{X}_U, \Q_\ell(i))^{G_{\F_q}}
            \end{gather*}
            is an isomorphism, and
            \item the eigenvalue $1$ of $\Frob_q$ acting on $H^{2i}(\ol{X}_U, \Q_\ell(i))$ is semisimple,
        \end{enumerate}
        then the Beilinson--Bloch conjecture holds in codimension $i$ for $X_K$.
	\end{thm}
    It is useful to translate the hypotheses of Theorem \ref{thmBBCriterionIntro} into statements about $X$, rather than $X_U$, and this is done in Corollary \ref{corClosedPointsBBCriterion}. Later in this introduction, we will discuss examples where these hypotheses can be checked in practice, using a method developed by Soul\'e \cite{Soule1984}, Geisser \cite{Geisser1998}, and Kahn \cite{Kahn2003} to unconditionally prove cases of Beilinson's conjecture (see \cite[Theorem 76]{Kahn2005}, and also Lemma \ref{lemTateEquivs}).
    \par
    Theorem \ref{thmBBCriterionIntro} is deduced from an analogous statement which also applies over a higher-dimensional base (Theorem \ref{thmBBCritArbBase}). The proof is essentially a weight argument, based on the Leray spectral sequence and Deligne's Weil II \cite{Deligne1980}, plus some consequences of the latter developed in \cite{Jannsen1990}. 
    \begin{rem}\label{remJannsenContrast}
        Theorem \ref{thmBBCriterionIntro} should be compared with the related \cite[Theorem 12.16]{Jannsen1990} and \cite[Theorem 5.3]{Jannsen2007}. It mainly differs in that Jannsen's results require an additional condition on fibers of $f$ over closed points of $U$, and are formulated for more general motivic cohomology groups, rather than just Chow groups. By contrast, the proof of Theorem \ref{thmBBCriterionIntro} exploits cohomological purity to dispense with additional hypotheses on closed fibers, yielding a statement which only concerns Chow groups. A practical advantage of our results, relative to Jannsen's, comes from the variant of the theorem allowing for non-smooth morphisms (Corollary \ref{corClosedPointsBBCriterion}). This is obtained by combining Theorem \ref{thmBBCriterionIntro} with properties of the weight filtration on $\ell$-adic homology developed in \cite{Jannsen1990}, and its assumptions are often more straightforward to directly check in examples than Jannsen's. Also, we explicitly connect our results to $L$-functions, whereas Jannsen's results are phrased in the related language of Abel--Jacobi maps. 
        \par
        For all of the concrete varieties $V$ for which we prove the Beilinson--Bloch conjecture in this paper, we also show that the hypotheses of \cite[Theorem 5.3]{Jannsen2007} are satisfied. This theorem has consequences for all motivic cohomology groups of $V$, and also implies that $V$ satisfies Murre's conjecture (loc. cit. Conjecture 4.1). Theorem \ref{thmBBCriterionIntro} is used in these examples to obtain the $L$-function formulation of the Beilinson--Bloch conjecture. In a follow-up paper \cite{Broe2025}, we use Theorem \ref{thmBBCriterionIntro} to prove cases of the Beilinson--Bloch conjecture for varieties which we do not know to verify the assumptions of \cite[Theorem 5.3]{Jannsen2007}. See Remark \ref{remJannsenContrastDetailed} for more details. 
    \end{rem}
	
	\par
	Our first application of Theorem \ref{thmBBCriterionIntro} is a new proof of the following result of Geisser, which connects the Tate and Birch and Swinnerton-Dyer conjectures.
	
	\begin{thm}[Theorem \ref{thmATGeneralization}]\label{thmATGeneralizationIntro}
		Let $f: X \to C$ be a flat projective morphism, with $X$ a smooth variety over $\F_q$, and $X_K$ smooth and geometrically integral over $K$. The following are equivalent:
		\begin{enumerate}[(i)]
			\item the Tate conjecture holds for divisors on $X$;
			\item the Tate conjecture holds for divisors on $X_K$, and the Birch and Swinnerton-Dyer conjecture holds for the Albanese variety $\Alb(X_K)$;
			\item the Brauer group of $X$ is finite;
			\item the Tate conjecture holds for divisors on $X_K$, and the Tate--Shafarevich group of $\Alb(X_K)$ is finite.
		\end{enumerate}
	\end{thm}
	In the case where $X$ is a surface, the equivalence of (iii) and (iv) is due to Artin--Grothendieck \cite[\hphantom{}(4.7)]{Grothendieck1968}, and is closely linked with a statement conjectured by Artin--Tate in the 1960s \cite[Conjecture (d)]{Tate1966b}, which relates the critical value of the zeta function of $X$ with that of the $L$-function of $\Alb(X_K)$. That conjecture was originally proved in the early 2000s by combining Artin--Grothendieck's theorem with deep results of Milne \cite{Milne1986} and Kato--Trihan \cite{KT2003}. See \cite{LRS2022} for further background, and two alternative proofs of Artin--Tate's Conjecture (d). 
    \par
    For $X$ of arbitrary dimension, the equivalence of (iii) and (iv) is \cite[Theorem 1.1]{Geisser2021} (with the caveat that we require $X$ to be projective over $\F_q$, while Geisser only requires it to be proper). Our proof is different from Geisser's, which uses \'etale motivic cohomology, and also yields a formula for the order of the Tate--Shafarevich group of $\Alb(X_K)$, conditional on certain conjectures. We first show the equivalence of (i) and (ii): the other equivalences then immediately follow from \cite{Milne1986} and \cite{KT2003}. 
	\par
	Theorem \ref{thmATGeneralizationIntro} was previously reproved (and generalized) in \cite{Qin2024}, via some geometric arguments with Brauer groups. All three proofs use the aforementioned results of Milne and Kato--Trihan. As noted earlier, our method is otherwise based on classical weight techniques, and makes almost no direct reference to Brauer or Tate--Shafarevich groups. In this sense, it yields a simpler proof, at the cost of not recovering the finer consequences of the other proofs for the structure of these groups.
    \par
	Following the demonstration of Theorem \ref{thmATGeneralizationIntro}, we proceed to use Theorem \ref{thmBBCriterionIntro} to prove some cases of the Beilinson--Bloch conjecture in codimension $> 1$. First, we give an alternative proof of a result of Kahn \cite[Corollary 1]{Kahn2021}, which asserts the finiteness of the Albanese kernel for certain constant surfaces over function fields. To go further, we need to establish some new cases of the Tate conjecture.
	
	\subsection{Powers of CM elliptic curves}
	For an abelian variety $A$ over a field $k$, let $\End_k^0(A) = \End_k(A) \otimes_\Z \Q$ denote its endomorphism algebra. An elliptic curve $E$ over $k$ is said to have complex multiplication (defined over $k$) if $\dim_\Q \End_k^0(E) \ge 2$.
	We adopt the convention that $E$ is ordinary\footnote{According to this definition, every elliptic curve in characteristic zero is ordinary.} if $\dim_\Q \End_{\ol{k}}^0(\ol{E}) \le 2$, supersingular if $\dim_\Q \End_{\ol{k}}^0(\ol{E})= 4$, and arithmetically supersingular if $\dim_\Q \End_k^0(E) = 4$. For a CM elliptic curve which is not arithmetically supersingular, the endomorphism algebra $\End^0_k(E)$ is an imaginary quadratic field. The endomorphism algebra of an arithmetically supersingular elliptic curve is a quaternion algebra over $\Q$ \cite[Corollary III.9.4]{Silverman2009}. 
	\par
	Our original motivation for writing this paper was an attempt to understand the Chow groups of $E^g$, where $E$ is an ordinary CM elliptic curve over $k$ and $g > 0$. We will describe a decomposition of the Chow motive of $E^g$ which gives considerable insight into its Chow groups. Powers of $E$ are thus natural test objects for conjectures on algebraic cycles: their Chow groups are interesting, yet more structured than those of most varieties. Indeed, Bloch provided early evidence for the Beilinson--Bloch conjecture by studying the Jacobian of the Fermat quartic curve over $\Q$ \cite{Bloch1984}, which becomes isogenous to the cube of a CM elliptic curve after base change to $\Q(e^{\pi i / 4})$. Tate also proved his eponymous conjecture for powers of CM elliptic curves over number fields \cite[page 106]{Tate1965}. 
	\par
	The decomposition is obtained through a general procedure. For any abelian variety $A$ over a field $k$, the action of the endomorphism algebra $D = \End^0_k(A)$ induces a direct sum decomposition of the motive $\mf{h}(A)$ of $A$. From this perspective, Deninger--Murre's canonical Chow--K\"unneth decomposition \cite[Corollary 3.2]{DM1991} is induced by the action of $\Q \subseteq D$. The Deninger--Murre decomposition is refined by a construction of Moonen \cite[Theorem 1]{Moonen2014}, which decomposes $\mf{h}(A)$ into subrepresentations of $D^{\op, \times}$.
	\par
	Particularly when $D$ is large, one can sometimes gain very fine control over the motives appearing in Moonen's decomposition, and over their Chow groups. For instance, Fu--Li showed that the motive of a supersingular abelian variety $A$ over an algebraically closed field has a very simple decomposition \cite[Theorem 2.9]{LF2021}. Among other applications, they used this to reprove some earlier results of Fakhruddin, including that algebraic and numerical equivalence coincide on $A$ up to torsion (\cite[Corollary 2.12(iii)]{LF2021}, \cite[Remark 1]{Fakhruddin2002}).
	\par
	Any supersingular abelian variety over an algebraically closed field is isogenous to a power of a supersingular elliptic curve \cite[Theorem (4.2)]{Oort1974}, and a power of an arithmetically supersingular elliptic curve over an arbitrary field has such a decomposition of its motive (see the treatment of \cite[Section 2]{Moonen2024a}, which applies in this generality). It is thus natural to ask whether the motive of a power of an elliptic curve $E$ with $F = \End^0_k(E)$ an imaginary quadratic field might be similarly manageable. However, the smaller endomorphism algebra introduces complications, so that interesting new motives appear in the decomposition of $\mf{h}(E^g)$. These are Tate twists of motives of the form $\otimes_F^r \mf{h}^1(E)$.
	\par
	Note that this explicit decomposition essentially appears in \cite[Example 2.8]{Cao2018}, though Cao only states it in characteristic zero. The decomposition was also discussed in \cite{Moonen2024a}. We give a proof of its existence and form over an arbitrary field (Corollary \ref{corQTensorPowerCMDecomp}).
	\par
	By work of Imai \cite{Imai1976}, Spie\ss \ \cite{Spiess1999}, Lombardo \cite{Lombardo2016}, and Kahn \cite{Kahn2025}, the Tate conjecture is known for products of elliptic curves over finitely generated fields. For a power of a CM elliptic curve, we use the decomposition of the motive to obtain a stronger result.
	\begin{thm}[Corollary \ref{corTateForEg}]\label{thmTate}
		For $E$ a CM elliptic curve over a finitely generated field $k$, $C$ a smooth projective curve over $k$, and $g \ge 0$, the full Tate conjecture holds for $E^g \times_k C$. 
	\end{thm}
    The proof of Theorem \ref{thmTate} is a specialization argument reducing to the case of a finite base field, where the statement follows from arithmetic properties of the zeta function of $E$ (similarly to the method of \cite{Spiess1999}). 
    \par
    From Theorems \ref{thmBBCriterionIntro} and \ref{thmTate}, we deduce the Beilinson--Bloch conjecture for a power of a CM elliptic curve $E$ over a global function field.  As a consequence, we find that the Chow groups of $E^g$ are quite simple. 
	\begin{thm}[Theorem \ref{thmChVanishOverFuncField}]\label{thmChVanishOverFuncFieldIntro}
		When $K$ is a global function field, and $E$ is an ordinary CM elliptic curve over $K$ with endomorphism algebra $F$, we have $\CH^i(\otimes_F^r \mf{h}^1(E)) = 0$ for $r \ge 2$ and all $i \in \Z$.
	\end{thm}

    \begin{rem}
        Jannsen previously used known cases of the Tate conjecture to deduce related results for the generic fiber of $\pi: X \to C$, where $X$ is a product of elliptic curves and copies of $\mathbb{P}^1$ over a finite field, and $\pi$ is the projection onto one of the factors \cite[Remarks 5.2]{Jannsen2007}. 
    \end{rem}
	\par
	In view of Theorem \ref{thmChVanishOverFuncFieldIntro}, it is notable that the Chow groups of powers of CM elliptic curves can behave very differently in characteristic zero. For an ordinary CM elliptic curve $E$ over an arbitrary field $k$, we show that $\CH^2_{\alg}(\otimes_F^3 \mf{h}^1(E))$ (the Chow group modulo algebraic equivalence) is the numerical Griffiths group $\Griff(E^3)$, defined to consist of numerically trivial 1-cycles in $E^3$ modulo algebraically trivial ones (with rational coefficients). Bloch's study of the Jacobian of the Fermat quartic then shows that this group can be nonzero over a number field.
	\par
	In \cite[page 9]{Moonen2024a}, Moonen asked whether for an ordinary CM elliptic curve $E$ over an arbitrary field $k$, we have that $\CH^i(\otimes_F^g \mf{h}^1(E)) = 0$ for all $g > 0$ and all $i \neq g$. As discussed, the answer to this question is negative in characteristic zero. In the case where $k$ has transcendence degree $\le 1$ over $\F_p$, Theorem \ref{thmChVanishOverFuncFieldIntro} quickly implies that the answer is positive. We show in Corollary \ref{corGeneralConcentrationOfCH} that if one assumes the agreement of rational and homological equivalence (up to torsion) on smooth projective varieties over $\F_p$, then the Chow groups in question vanish when $k$ is any field of characteristic $p$. If $\CH^i(\otimes_F^g \mf{h}^1(E)) = 0$ for all $g > 0$ and all $i \neq g$, then algebraic and numerical equivalence agree up to torsion on powers of $E$ (Corollary \ref{corAlgEqualsNumForEg}).
	
	\subsection*{Structure of the paper}
	
	In Section 2 of the paper we review some well-known material on Chow groups and motives. Section 3 states Jannsen's generalizations of the conjectures of Tate and Beilinson, and proves some sufficient conditions which imply them. In Section 4 we prove Theorem \ref{thmBBCriterionIntro}, then collect some of its consequences (including Geisser's theorem). Our arguments here are largely based on an analysis of the Leray spectral sequence, using Deligne's $\ell$-adic decomposition theorem for smooth projective maps (recalled in Lemma \ref{lemDecomp}) and the weight filtration on $\ell$-adic homology constructed in \cite{Jannsen1990}.
	\par
	The rest of the paper is concerned with powers of CM elliptic curves. Section 5 reviews some general facts about such curves. Section 6 constructs the decomposition of the motive $\mf{h}(E^g)$, for $E$ an ordinary elliptic curve with CM by an imaginary quadratic field $F$, and uses it to prove Theorem \ref{thmTate}. The Beilinson--Bloch conjecture for a power of a CM elliptic curve over a global function field follows shortly. This section also proves a couple of general results on the structure of the motivic decomposition, which hold over an arbitrary field: we show that the decomposition satisfies a certain uniqueness property, and that $\CH^\bullet(\otimes_F^g \mf{h}^1(E))$ consists of homologically trivial cycles for all $g \ge 1$. We then treat Moonen's question.
	\par
	In an appendix, we compute the transcendental motives of the square and cube of an ordinary CM elliptic curve $E$ over an arbitrary field, in the sense of \cite{Kahn2021}. These are certain summands of $\mf{h}(E^2)$ and $\mf{h}(E^3)$, which respectively split off the Albanese kernel and the Griffiths group from the Chow group. As a corollary, we find $\Griff(E^3) = \CH^2_{\alg}(\otimes_F^3 \mf{h}^1(E))$ (though this can also be seen more directly using the decomposition of $\mf{h}(E^3)$). We note another interesting application (Corollary \ref{corCH2t3OverFuncField}), which states that for an ordinary CM elliptic curve $E$ over a field $k$, and a smooth projective variety $X$ over $k$, the pullback map $\CH^2(\mf{h}(X) \otimes_\Q \otimes_F^3 \mf{h}^1(E)) \to \CH^2(\otimes_F^3 \mf{h}^1(E_{k(X)}))$ is an isomorphism.
	
	\subsection*{Acknowledgements}
	I am grateful to Ben Moonen for his course at the 2024 Arizona Winter School, and for generous correspondence and advice during the writing of this paper. I thank the fellow members of the AWS project group which was the impetus for it: Beatriz Barbero, Andrea Bourque, Michael Maex, Waleed Qaisar, Peikai Qi, and Sjoerd de Vries, and the project leads, Jef Laga and Ziquan Yang. I benefited from helpful discussions with Alexander Beilinson, Bruno Kahn, Kiran Kedlaya, and Shubhankar Sahai. I also thank Jennifer Balakrishnan, Alexander Bertoloni-Meli, David Rohrlich, Padma Srinivasan, and Jared Weinstein for thoughtful comments on earlier drafts. I am indebted to Jannsen's book \cite{Jannsen1990}, which inspired many of the ideas of this paper.
	
	\subsection*{Notation and conventions}
	
	For a field $k$, we let $\overline{k}$ and $k_s$ respectively denote choices of algebraic and separable closures of $k$. We set $G_k = \Gal(k_s / k)$. By $\ell$ we will always mean a fixed prime not equal to the characteristic of the base field. A variety over $k$ is an integral, separated, finite-type $k$-scheme.
	\par
	For schemes $X, \ Y, \ Z$ and morphisms $f: X \to Z$, $g: Y \to Z$, we will often denote by $f_Y: X_Y = X \times_Z Y \to Y$ the base change of $f$ along $g$ (with the dependence on $g$ left implicit). When $X, \ Z$, and $f$ are over a field $k$, we write $\ol{X}$, $\ol{Z}$ and $\ol{f}: \ol{X} \to \ol{Z}$ for their respective base changes to the algebraic closure of $k$. For a sheaf $\mc{F}$ on $X$, we write $\ol{\mc{F}}$ for the pullback to $\ol{X}$. When $X$ is integral, $k(X)$ denotes its function field.
	\par
	For a prime power $q$, a Weil $q$-number of weight $w$ is an algebraic number $\alpha$ whose image under each embedding $\Q(\alpha) \to \C$ has absolute value $q^{w / 2}$. A Weil $q$-integer of weight $w$ is a Weil $q$-number of weight $w$ which is an algebraic integer. For a nonzero polynomial $P(t)$ with coefficients in a field $L$, and $\beta \in \ol{L}$, we write $\ord\limits_{t = \beta} P(t)$ for the number of times $t - \beta$ divides $P(t)$.
	\par
	Our main references for intersection theory and motives are \cite{Fulton1998} and \cite{Andre2004}, respectively (see \cite[Tag 0FG9]{stacks-project} for a condensed reference). For a finite-type, separated scheme $X$ over a field $k$, and a commutative ring $R$, the notation $Z_i(X)_R$ denotes the free $R$-module on the $i$-dimensional closed subvarieties of $X$. When each connected component of $X$ is pure-dimensional, so that the codimension of a closed subvariety of $X$ is well-defined, $Z^i(X)_R$ denotes the free $R$-module on the codimension-$i$ closed subvarieties of $X$. By $\CH_i(X)_R$ (resp. $\CH^i(X)_R$) we mean the Chow group of $X$, defined to be $Z_i(X)_R$ (resp. $Z^i(X)_R$) modulo rational equivalence. For $i < 0$, $\CH_i(X)_R$ and $\CH^i(X)_R$ are defined to be zero. For a closed subvariety $V$ of $X$, we let $[V]$ denote its class in $Z_\bullet(X)_R$.
	\par
	When $X$ is a smooth projective variety, and $\sim$ an adequate equivalence relation, $\CH_i^\sim(X)_R$ denotes the quotient of $\CH_i(X)_R$ by $\sim$. We similarly define $\CH^i_\sim(X)_R$. By $\mc{P}(k)$ we mean the category of smooth projective schemes over $k$. By $\mc{M}_\sim(k)_R$ we mean the category of pure motives over $k$, modulo $\sim$ and with coefficients in $R$. For a morphism $f: X \to Y$ in $\mc{P}(k)$, we write $\Gamma_f$ for the graph of $f$ in $\CH^\bullet(X \times_k Y)$. For a cycle $c \in \CH^i(X \times_k X)$, we denote by $c^t$ the pullback of $c$ under the transpose map. This gives rise to a contravariant functor $X \mapsto \mf{h}(X): \mc{P}(k)^{\operatorname*{op}} \to \mc{M}_\sim(k)_R$. For a motive $M \in \mc{M}_{\sim}(k)_R$ and adequate equivalence relation $\sim'$ which is coarser than $\sim$, we write $M_{\sim'} \in \mc{M}_{\sim'}(k)_R$ for the reduction of $M$ modulo $\sim'$. For $n \ge 0$, we write $n \cdot M$ for the direct sum of $n$ copies of $M$.
	\par
	The adequate equivalence relations we consider are rational equivalence, algebraic equivalence, homological equivalence (with respect to $\ell$-adic cohomology), and numerical equivalence, denoted by the respective shorthands $\rat, \ \alg,$  $\hom,$ and $\num$. Note that the former three equivalence relations also make sense for cycles on any finite-type separated $k$-scheme $X$ (though when $X$ is not smooth, one must define homological equivalence via a cycle class map to \'etale homology, rather than cohomology), and numerical equivalence makes sense when $X$ is additionally assumed to be smooth and proper. We use the same notation $\CH_i^\sim(X)_R$ in these cases. When $\sim$ is rational equivalence, we typically omit the subscript $\sim$ from all notation relating to Chow groups and motives. Similarly, when $R = \Q$, we omit the subscript $R$. Note that this differs from  \cite{Fulton1998}, which drops the subscript $R$ when $R = \Z$, but as we hardly consider Chow groups with integral coefficients at all, our convention seems sensible. 
	\par
	For $X \in \mc{P}(k)$, the subgroup $\CH^i(X)_0 \subseteq \CH^i(X)$ is defined to be the kernel of the $\ell$-adic cycle class map $\CH^i(X) \to H^{2i}(\ol{X}, \Q_\ell(i))$. For $M \in \mc{M}(k)$, $\CH^i(M)_0$ is defined similarly. We write $\CH^i(X)_{0, \Z}$ for the kernel of the $\ell$-adic cycle class map to $H^{2i}(\ol{X}, \Q_\ell(i))$ evaluated on the Chow group with integral coefficients. We denote $\ell$-adic cohomology of varieties and motives by $H^i(\ol{X}, \Q_\ell)$ and $H^i(M, \Q_\ell)$, respectively. 
	\par
	When $X$ is a smooth, proper, and geometrically integral variety over $k$, we write $\Alb(X)$ for the Albanese variety of $X$, and $\Pic^0(X)_{\red}$ for the reduced subscheme of the connected component of the Picard scheme of $X$. For such $X$, $\Alb(X)$ and $\Pic^0(X)_{\red}$ are abelian varieties, dual to each other \cite[page 3]{Wittenberg2008}. By $\GrpSch(k)$ we mean the category of group schemes over $k$. 
    \par
	When $k \cong \F_q$ is finite, the symbol $\Frob_q \in G_{\F_q}$ denotes the geometric Frobenius sending $x \mapsto x^{1/q}$. The action of $\Frob_q$ on the $\ell$-adic cohomology of any $\F_q$-scheme $X$ agrees with the action via pullback of the Frobenius endomorphism $\Frob_X$ of $X$. For any correspondence $\pi$ between smooth projective schemes $X$, $Y$ over $k$, $\pi \circ \Frob_X = \Frob_Y \circ \pi$, so any motive $M \in \mc{M}_{\sim}(k)_R$ has a canonical Frobenius endomorphism \cite[Exercices 4.1.4.1]{Andre2004}.

	\section{Intersection theory and motives}
	We recall some well-known facts from intersection theory and the theory of motives which we will use in the body of the paper. The material on motives (starting in Subsection \ref{ssecCKDecomp}) is used sparingly until the sections on CM elliptic curves, so the reader who is only interested in the earlier parts of the paper can skim or skip it without much difficulty. Let $k$ be a field.
	\par
	\subsection{Properties of Chow groups}
	For $i: Z \to X$ a closed immersion of finite-type separated schemes over $k$, $j: U \to X$ the complementary open immersion, $R$ a commutative ring, and $n \in \Z$, there is a \enquote{localization} exact sequence
	\begin{gather*}
		\CH_n(Z)_R \xrightarrow[]{i_*} \CH_n(X)_R \xrightarrow[]{j^*} \CH_n(U)_R \to 0,
	\end{gather*}
	cf. \cite[Proposition 1.8]{Fulton1998}. 
	\par
	Let $Y$ be another finite-type separated $k$-scheme. Suppose that $Y$ is integral with generic point $\eta$, and $f: X \to Y$ is a $k$-morphism. In the colimit over all closed subschemes properly contained in $Y$, the localization exact sequence yields another exact sequence
	\begin{gather}\label{eqLocalization}
		\bigoplus\limits_{\substack{Z \subsetneq Y \textrm{ closed}}} \CH_n(X_Z)_R \to \CH_n(X)_R \to \colim\limits_{\emptyset \neq U \subseteq Y \textrm{ open}} \CH_n(X_U)_R \to 0.
	\end{gather}
	
	The next lemma is well-known. We omit the proof, which is elementary (see e.g. the related \cite[Exercise 1.13]{MVW2006}).
	
	\begin{lem}
		The map sending a closed subvariety of $X$ to its generic fiber induces an isomorphism
		\begin{gather*}
			\colim\limits_{\emptyset \neq U \subseteq Y \textrm{ open}} \CH_n(X_U)_R \cong \CH_{n - \dim Y}(X_\eta)_R.
		\end{gather*}
	\end{lem}
	
	We note some nice properties under certain conditions of the pullback map on Chow groups associated to a map of schemes.
	\begin{lem}\label{lemCHInjectsUnderBaseExt}
		Let $X$ be a finite-type, smooth, separated scheme over $k$. Let $\sim$ be either rational, algebraic, homological, or (if $X$ is proper) numerical equivalence.
		\begin{enumerate}[(i)]
			\item For $L / k$ an arbitrary field extension, the pullback map $\CH^i(X) \to \CH^i(X_L)$ is injective. 
			\item When $L / k$ is algebraic, the pullback $\CH^i_{\sim}(X) \to \CH_{\sim}^i(X_L)$ is injective. If $L / k$ is purely inseparable, the pullback map is an isomorphism modulo $\sim$.
			\item If $X$ is quasi-projective, and $f: Y \to X$ is a finite Galois cover, then $f^*: \CH^i_\sim(X) \to \CH^i_\sim(Y)$ is injective, and $f^*(\CH^i_\sim(X)) = \CH^i_\sim(Y)^{\Gal(Y / X)}$.
		\end{enumerate}
	\end{lem}
	\begin{proof}
		Claim (i) follows from the argument of \cite[Lemma 2.1]{LS2024}. This uses the refined Gysin homomorphism of \cite[Chapter 6.2]{Fulton1998}; since a closed immersion of finite-type smooth schemes is a regular embedding \cite[Appendix B.7.2]{Fulton1998}, the Gysin homomorphism is well-defined in the generality needed here. Claim (ii) follows from the argument of \cite[Tag 0FH8]{stacks-project}, together with the fact that if $L / k$ is purely inseparable, the map $X_L \to X$ is a universal homeomorphism. Claim (iii) follows from \cite[Example 1.7.6]{Fulton1998}.
	\end{proof}
	
	\begin{lem}\label{lemAdEqRelsOnDivisors}
		For a smooth, proper variety $X$ over a field $k$, the $\ell$-adic cycle class map is injective on ${\CH^1_{\alg}(X) \otimes_\Q \Q_\ell}$, and algebraic and numerical equivalence of divisors agree up to torsion, i.e. $\CH^1_{\alg}(X) = \CH^1_{\num}(X)$.
	\end{lem}
	\begin{proof}
		By Lemma \ref{lemCHInjectsUnderBaseExt}, we may assume that $k$ is algebraically closed. The N\'eron-Severi group $\NS(X) = \CH^1_{\alg}(X)_{\Z}$ of $X$ is finitely generated by \cite[Expos\'e XIII, Th\'eor\`eme 5.1]{SGA6}. As $\Pic^0(X)(k)$ is $\ell$-divisible, the Kummer sequence implies that the cycle class map induces an injection
		\begin{gather*}
			\NS(X) \otimes_\Z \Z_\ell \to H^2(X, \Z_\ell(1)).
		\end{gather*}
		That $\CH^1_{\alg}(X) = \CH^1_{\num}(X)$ is \cite[Expos\'e XIII, Th\'eor\`eme 4.6]{SGA6}.
	\end{proof}
	
	\subsection{Chow--K\"unneth decompositions}\label{ssecCKDecomp}
	\begin{defn}
		A \textbf{Chow--K\"unneth decomposition} of a motive $M \in \mc{M}(k)$ (with respect to a classical Weil cohomology theory $H$) consists of an isomorphism 
        \begin{gather*}
            M \cong \bigoplus\limits_{i \in \Z} M_i
        \end{gather*}
        such that for all $i, j \in \Z$ with $i \neq j$, $H^j(M_i) = 0$.
	\end{defn}
	See \cite[\hphantom{}3.4]{Andre2004} for the definition of a classical Weil cohomology theory. By Th\'eor\`eme 4.2.5.2 and Remarque 5.1.1.2 of loc. cit., the definition does not depend on the choice of $H$, so we will simply speak of Chow--K\"unneth decompositions, with no reference to $H$.
	\par
	When $X$ is a smooth projective scheme over $k$ of dimension $d$, we denote the summands in a Chow--K\"unneth decomposition of $\mf{h}(X)$ via 
	\begin{gather*}
		\mf{h}(X) \cong \bigoplus\limits_{i=0}^{2d} \mf{h}^i(X).
	\end{gather*}
    Deninger--Murre showed in \cite[Proposition 3.3]{DM1991} that motives of abelian varieties possess a canonical Chow--K\"unneth decomposition, such that when $X,\ Y$ are abelian varieties over $k$, the canonical identification
	\begin{gather*}
		\mf{h}(X) \otimes_\Q \mf{h}(Y) \cong \mf{h}(X \times_k Y)
	\end{gather*}
	induces an isomorphism
	\begin{gather*}
		\bigoplus\limits_{i + j = n} \mf{h}^i(X) \otimes_\Q \mf{h}^j(Y) \cong \mf{h}^n(X \times_k Y). 
	\end{gather*}
	\par
	For any smooth projective variety $X$ over $k$ of dimension $d$, and any ample divisor $\xi \in \CH^1(X)$,  there exist orthogonal projectors $\pi_0,\ \pi_1,\ \pi_{2d-1},\ \pi_{2d} \in \CH^d(X \times_k X)$, which when $d \ge 2$ determine a decomposition 
	\begin{gather}\label{eqPartialCKDecomp}
		\mf{h}(X) \cong \mf{h}^0(X) \oplus \mf{h}^1(X) \oplus M \oplus \mf{h}^{2d-1}(X) \oplus \mf{h}^{2d}(X),
	\end{gather}
	as constructed as in \cite[\hphantom{}1.13 and Section 4]{Scholl1994} (here we use the zero-cycle $\xi^d$ as input to the construction of loc. cit. 1.13). When $d = 1$, so $X$ is a curve, the same construction yields a Chow--K\"unneth decomposition
	\begin{gather*}
		\mf{h}(X) \cong \mf{h}^0(X) \oplus \mf{h}^1(X) \oplus \mf{h}^{2}(X).
	\end{gather*}
	When $X$ is geometrically integral, the decomposition \eqref{eqPartialCKDecomp} satisfies $\mf{h}^1(X) \cong \mf{h}^1(\Pic^0(X)_{\red})$, where $\Pic^0(X)_{\red}$ is the Picard variety and $\mf{h}^1(\Pic^0(X)_{\red})$ is the degree 1 component of its canonical Chow--K\"unneth decomposition.
	
	\subsection{Motivic algebra}
	Let $X,\ Y$ be smooth projective varieties over $k$, with $X$ geometrically integral. Form decompositions of $\mf{h}(X)$ and $\mf{h}(Y)$ of the form \eqref{eqPartialCKDecomp}, and let $\mf{h}^1(X)$, $\mf{h}^1(Y)$ be their respective degree one components. The next lemma is a partial form of \cite[Lemma 2.1]{Kahn2021}. 
	\begin{lem}\label{lemMotiveCalcs}
		The following are true:
		\begin{enumerate}[(i)]
			\item $\CH^1(\mf{h}^1(X)) = \CH^1(X)_0 \cong \Pic^0(X)(k) \otimes_\Z \Q$;
			\item When $Y$ is geometrically integral, for any adequate equivalence relation $\sim$,
			\begin{gather*}
				\Hom_{\mc{M}_\sim(k)}(\mf{h}^1(X), \mf{h}^1(Y)) \cong \Hom_{\GrpSch(k)}(\Pic^0(X)_{\red}, \Pic^0(Y)_{\red}) \otimes_\Z \Q;
			\end{gather*}
			\item As abelian groups, $\Hom_{\mc{M}(k)}(\mf{h}^1(X), \mf{h}(Y)) \cong \Hom_{\mc{P}(k)}(Y, \Pic^0(X)_{\red}) \otimes_\Z \Q$.
		\end{enumerate}
	\end{lem}
	In particular, $A \mapsto \mf{h}^1(A)$ is a fully faithful contravariant functor from the isogeny category of abelian varieties over $k$ to $\mc{M}_\sim(k)$. 
	
	\subsection{Motives of abelian type}
	A motive is said to be of abelian type if it belongs to the thick and rigid tensor subcategory of $\mc{M}(k)$ spanned by motives of abelian varieties over $k$. Any summand or Tate twist of the motive of an abelian variety is of abelian type. Motives of abelian type are finite-dimensional in the sense of Kimura \cite[Example 9.1]{Kimura2005}, and have Chow--K\"unneth decompositions \cite[Corollary 1.6]{Ancona2017}.
	
	\subsection{Conservativity and nilpotence}
	
	Fix a classical Weil cohomology theory $H$ on smooth projective varieties over $k$. This gives rise to a realization functor $R$ on $\mc{M}(k)$, such that $H = R \circ \mf{h}$. \par
    We will need some results from the theory of finite-dimensional motives \cite[Definition 3.7]{Kimura2005}. Let $M \in \mc{M}(k)$ be a motive, and $f: M \to M$ an endomorphism. The following lemma is immediately implied by \cite[Corollary 7.9]{Kimura2005}.
	
	\begin{lem}\label{conservativity}
		 If $M$ is finite-dimensional, and $R(f)$ is an isomorphism, then so is $f$. 
	\end{lem}
	
	\begin{cor}\label{realDetectsZero}
		If $M$ is finite-dimensional, and $R(M) = 0$, then $M = 0$.
	\end{cor}
	
	We recall Kimura's nilpotence theorem, which is a powerful tool to lift information from numerical equivalence to rational equivalence. 
	
	\begin{thm}[{\cite[Proposition 7.5]{Kimura2005}}]\label{thmNilpotence}
		If $M$ is finite-dimensional, and $f$ induces the zero map modulo numerical equivalence, then $f$ is nilpotent.
	\end{thm}
	
	\begin{cor}
		If $M$ is finite-dimensional, and $R(f) = 0$, then $f$ is nilpotent.
	\end{cor}
	
	\begin{cor}[{\cite[Corollary 7.8]{Kimura2005}}]\label{corDecompsOfFinDimMotivesLift}
		If $M$ is finite-dimensional, then any direct sum decomposition of $M$ modulo homological equivalence (with respect to $H$) lifts non-canonically to a decomposition modulo rational equivalence. 
	\end{cor}

    Vial proved that finite-dimensionality of a motive can be checked after extension of scalars.
    \begin{prop}[{\cite[Theorem 2]{Vial2017}}]\label{propFdDescent}
        For any field extension $L / k$, $M$ is finite-dimensional iff $M_L$ is finite-dimensional.
    \end{prop}

    We note for later use Ancona's result that the $\ell$-adic realization functor on the category of motives of abelian type over a finite field is conservative, i.e. detects isomorphisms.

    \begin{prop}[{\cite[Theorem 0.4]{Ancona2017}}]\label{propConservativity}
        Let $g: M \to N$ be a morphism of motives of abelian type over $\F_q$. If $g$ induces an isomorphism on $\ell$-adic cohomology, then $g$ is an isomorphism.
    \end{prop}

	\subsection{Motives with coefficients in $F$}
	Let $F$ be an imaginary quadratic field, and $\sim$ an adequate equivalence relation. We may consider $\mc{M}_\sim(k)_F$ as the category of objects with $F$-action in $\mc{M}_\sim(k)_\Q$. There is thus a forgetful functor $\Phi: \mc{M}_\sim(k)_F \to \mc{M}_\sim(k)_\Q$, such that $\Phi(\mathbf{1}_F)$ is (non-canonically) isomorphic to $2 \cdot \mathbf{1}_\Q$. The conjugate of $M \in \mc{M}_\sim(k)_F$ is defined to be the motive $\overline{M}$ such that $\Phi(M) = \Phi(\overline{M})$, and for $x \in F$, multiplication by $x$ on $M$ acts as multiplication by $\overline{x}$ on $\overline{M}$. 
	\par
	For $M_1, M_2 \in \mc{M}_\sim(k)_F$, there is a natural action of $F \otimes_\Q F$ on $M_1 \otimes_\Q M_2$. Let $r \in F$ satisfy $r^2 = -D$ for a positive integer $D$. The element
	\begin{gather*}
		\frac{1}{2}(1 \otimes_\Q 1 - \frac{1}{D}(r \otimes_\Q r)) \in F \otimes_\Q F
	\end{gather*}
	gives an idempotent projector $M_1 \otimes_\Q M_2 \to M_1 \otimes_\Q M_2$, which defines a decomposition in $\mc{M}_\sim(k)_F$  
	\begin{equation}\label{productDecompOverK}
		M_1 \otimes_\Q M_2 \cong M_1 \otimes_F M_2 \oplus M_1 \otimes_F \overline{M_2}. 
	\end{equation}
	Our convention here and throughout is that when $M_1 \otimes_\Q M_2$ is considered as an object of $\mc{M}_\sim(k)_F$, the $F$-action is inherited from the first tensor factor: in particular, $M_1 \otimes_\Q M_2$ is not necessarily isomorphic to  $M_2 \otimes_\Q M_1$ in $\mc{M}_\sim(k)_F$. The decomposition \eqref{productDecompOverK} does not depend on the choice of $r$.
	\par
	The tensor operations $\otimes_F$ and $\otimes_\Q$ associate with each other: for $M_1, M_2, M_3 \in \mc{M}_\sim(k)_F$, there is a natural isomorphism
	\begin{gather*}
		M_1 \otimes_\Q (M_2 \otimes_F M_3) \cong (M_1 \otimes_\Q M_2) \otimes_F M_3
	\end{gather*}
	in $\mc{M}_\sim(k)_F$.
	
	\section{The generalized conjectures of Tate and Beilinson}\label{secGenConjs}
	For $k$ a field, and $f: X \to \Spec k$ a separated morphism of finite type, define the $\ell$-adic \'etale homology of $X$ via
	\begin{gather*}
		H_i(\ol{X}, \Q_\ell(j)) = H^{-i}(\ol{X}, (R{\ol{f}}^! \Q_\ell)(-j)).
	\end{gather*}
	Weight arguments with \'etale homology will figure prominently in the proofs of our main results. 
	\begin{prop}\label{propHomologyProps}
		\'Etale homology has the following features.
			\begin{enumerate}[(i)]
			\item There is a cycle class map
			\begin{gather*}
				\cl: Z_i(X) \to H_{2i}(\ol{X}, \Q_\ell(i))
			\end{gather*}
			which factors through $\CH_i(X)$ \cite[Remark 7.4]{Jannsen1990}. 
			\item When $X$ is a smooth variety of dimension $d$ over $k$, and $i, j \in \Z$, there is a canonical, $G_k$-equivariant Poincar\'e duality isomorphism
			\begin{gather}\label{eqSmoothHoCoIso}
				H^{2d-i}(\ol{X}, \Q_\ell(d-j)) \cong H_i(\ol{X}, \Q_\ell(j))
			\end{gather}
			compatible with cycle class maps \cite[Definition 6.1 j)]{Jannsen1990}.
			\item For $Z \hookrightarrow X$ a closed immersion, $U = X - Z$, and $j \in \Z$, there is a long exact sequence
			\begin{gather*}
				... \to H_{i+1}(\ol{U}, \Q_\ell(j)) \to H_i(\ol{Z}, \Q_\ell(j)) \to H_i(\ol{X}, \Q_\ell(j)) \to H_i(\ol{U}, \Q_\ell(j)) \to...,
			\end{gather*}
			cf. \cite[Definition 6.1 k)]{Jannsen1990}.
			\item When $k$ is finitely generated, the \'etale homology of $X$ has an ascending, Galois-equivariant weight filtration $W_\bullet H_i(\ol{X}, \Q_\ell(j))$ \cite[Lemma 6.8.2]{Jannsen1990}. When $k \cong \F_q$ is finite and $m \in \Z$, the eigenvalues of $\Frob_q$ acting on the $m$-th graded component $\gr^W_m H_i(\ol{X}, \Q_\ell(j))$ are Weil $q$-numbers of weight $m$.
			\item For $k$ a finitely generated field, $j: U \hookrightarrow X$ an open immersion with closed complement $Z$, and $a, b \in \Z$, the pullback map
			\begin{gather*}
				j^*: W_{2b-a} H_{a}(\ol{X}, \Q_\ell(b)) \to  W_{2b-a} H_{a}(\ol{U}, \Q_\ell(b))
			\end{gather*}
			is surjective \cite[Lemma 7.5]{Jannsen1990}. Taking $k \cong \F_q$, $a = 2i$ and $b = i$ for some $i \in \Z$ (and appealing to Lemma \ref{lemFrobSim1Exact}), this yields a commutative diagram with exact rows
			\begin{equation*}
				\begin{tikzcd}
					\CH_i(Z) \otimes_\Q \Q_\ell \arrow[r] \arrow[d] & \CH_i(X) \otimes_\Q \Q_\ell \arrow[r] \arrow[d] & \CH_i(U) \otimes_\Q \Q_\ell \arrow[r, overlay] \arrow[d] & 0 \\
					H_{2i}(\ol{Z}, \Q_\ell(i))^{\Frob_q \sim 1} \arrow[r] & H_{2i}(\ol{X}, \Q_\ell(i))^{\Frob_q \sim 1} \arrow[r] & H_{2i}(\ol{U}, \Q_\ell(i))^{\Frob_q \sim 1} \arrow[r, overlay] & 0.
				\end{tikzcd}
			\end{equation*}
			Here the superscript $\Frob_q \sim 1$ denotes the generalized eigenspace of the eigenvalue 1 of $\Frob_q$.
		\end{enumerate}
	\end{prop} 
	Part of the above proposition is restating properties of a twisted Poincar\'e duality theory with weights \cite[Definitions 6.1 and 6.5]{Jannsen1990}; \'etale homology and cohomology constitute such a theory \cite[Lemma 6.8.2]{Jannsen1990}.
	\par
	Using \'etale homology, Jannsen generalized the conjectures of Tate and Beilinson to arbitrary varieties, as recalled below. Here, for a $G_{\F_q}$-representation on a finite-dimensional $\Q_\ell$-vector space $V$, we define
	\begin{gather*}
		\chi(V, t) = \det(1-\Frob_q t | V).
	\end{gather*}
	
	\begin{conj}[$T_i(X)$: {\cite[Conjecture 7.3]{Jannsen1990}}]\label{conjGenTate}
		For $X$ a variety over a finitely generated field $k$, the cycle class map
		\begin{gather*}
			\CH_i(X) \otimes_{\Q} \Q_{\ell} \to H_{2i}(\ol{X}, \Q_{\ell}(i))^{G_k}
		\end{gather*}
		is surjective. 
	\end{conj}
	
	\begin{conj}[$B_i(X)$: {\cite[Conjectures 12.4 a) and 12.6 b)]{Jannsen1990}}]\label{conjGenBeilinson}
		For $X$ a variety over a finite field $k \cong \F_q$,
		\begin{enumerate}[(i)]
			\item the cycle class map induces an isomorphism
			\begin{gather*}
				\CH_i(X) \otimes_{\Q} \Q_{\ell} \cong H_{2i}(\ol{X}, \Q_\ell(i))^{G_k};
			\end{gather*}
			\item ($SS_i(X)$) the eigenvalue 1 of $\Frob_q$ acting on $H_{2i}(\ol{X}, \Q_\ell(i))$ is semisimple. 
		\end{enumerate} 
		In particular, 
		\begin{gather*}
			\dim_\Q \CH_i(X) = \ord_{t=1} \chi(H_{2i}(\ol{X}, \Q_\ell(i)), t).
		\end{gather*}
	\end{conj}
	
	\begin{rem}
		The conjectures \cite[Conjectures 12.4 a) and 12.6 b)]{Jannsen1990} are more general: they concern a map from the motivic homology group $H_r^{\mc{M}}(X, \Q(s)) \otimes_\Q \Q_\ell$ to $H_r(\ol{X}, \Q_\ell(s))$. The conjecture $B_i(X)$ is the case $(r, s) = (2i, i)$.
	\end{rem}

	The conjectures also make sense for $X$ a finite-type separated scheme over $k$. When $k$ is finitely generated, we write $T_i(X)$ for the truth value of Conjecture \ref{conjGenTate} for such $X$ and $i \in \Z$. When $k$ is finite, we write $B_i(X)$ for the truth value of Conjecture \ref{conjGenBeilinson} for $X$ and $i$, and $SS_i(X)$ for the truth of Conjecture \ref{conjGenBeilinson}(ii). When all connected components of $X$ are pure-dimensional, we write $T^i(X)$ (resp. $B^i(X)$) for the truth of Conjecture \ref{conjGenTate} (resp. Conjecture \ref{conjGenBeilinson}) for cycles of codimension $i$ in $X$, and similarly write $SS^i(X)$ for Conjecture \ref{conjGenBeilinson}(ii) in complementary dimension. These forms of the conjectures also apply to $M$ a motive in $\mc{M}(k)$, and we use the same notation $T^i(M), \ B^i(M), \ SS^i(M)$. We caution that in all discussion of these conjectures, it is important to keep track of the base field one is working over (which is left implicit in the notation).
    \par
    We now give some sufficient conditions for the conjectures. In the remainder of the paper, we will often implicitly use the following lemma without further comment. 
    \begin{lem}\label{lemFrobSim1Exact}
        The functor sending a finite-dimensional continuous $G_{\F_q}$-representation $V$ over $\Q_\ell$ to $V^{\Frob_q \sim 1}$ is exact.
    \end{lem}
    \begin{proof}
        Given a short exact sequence of such representations
        \begin{gather*}
            0 \to V_1 \to V_2 \to V_3 \to 0,
        \end{gather*}
        write $\det(t - \Frob_q | V_2) = (t-1)^r P(t)$, where $P(1) \neq 0$. For each $i$, the surjection
        \begin{gather*}
            P(\Frob_q): V_i \to V_i^{\Frob_q\sim 1}
        \end{gather*}
        is split by the inclusion. We thus obtain a commutative diagram with exact rows
        \[
		\begin{tikzcd}
			0 \arrow[r] & V_1 \arrow[r] \arrow[d] & V_2 \arrow[r] \arrow[d] & V_3 \arrow[r] \arrow[d] & 0 \\ 
            0 \arrow[r] & V_1^{\Frob_q \sim 1} \arrow[r] & V_2^{\Frob_q \sim 1} \arrow[r] & V_3^{\Frob_q \sim 1} \arrow[r] & 0,
		\end{tikzcd}
		\]
        where each vertical arrow is the map $P(\Frob_q)$.
    \end{proof}

	\begin{lem}\label{lemTateForOpen}
		Let $Z \hookrightarrow X$ be a closed immersion of finite-type separated schemes over a finite field $k \cong \F_q$, $U = X - Z$, and $i \in \Z$. Then
		\begin{enumerate}[(i)]
			\item $T_i(Z)$, $SS_i(Z)$ and $B_i(X)$ together imply $B_i(U)$;
			\item $T_i(Z)$, $SS_i(Z)$, and $T_i(U)$ together imply $T_i(X)$.
		\end{enumerate}
	\end{lem}
	\begin{proof}
		Consider the commutative diagram with exact rows
		\[
		\begin{tikzcd}
			\CH_i(Z) \otimes_\Q \Q_\ell \arrow[r] \arrow[d] & \CH_i(X) \otimes_\Q \Q_\ell \arrow[r] \arrow[d] & \CH_i(U) \otimes_\Q \Q_\ell \arrow[r, overlay] \arrow[d] & 0 \\
			H_{2i}(\ol{Z}, \Q_\ell(i))^{\Frob_q \sim 1} \arrow[r] & H_{2i}(\ol{X}, \Q_\ell(i))^{\Frob_q \sim 1} \arrow[r] & H_{2i}(\ol{U}, \Q_\ell(i))^{\Frob_q \sim 1} \arrow[r, overlay] & 0.
		\end{tikzcd}
		\]
		Under the hypotheses of (i), the left vertical arrow is surjective, and the middle vertical arrow is an isomorphism. The cycle class map $\CH_i(U) \otimes_\Q \Q_\ell \to H_{2i}(\ol{U}, \Q_\ell(i))^{\Frob_q \sim 1}$ is injective by the four-lemma, and thus an isomorphism for dimension reasons.
		\par
		In the situation of (ii), the hypothesis $SS_i(Z)$ implies that
        \begin{gather*}
            H_{2i}(\ol{Z}, \Q_\ell(i))^{\Frob_q \sim 1} = H_{2i}(\ol{Z}, \Q_\ell(i))^{G_{\F_q}}.
        \end{gather*}
        Likewise, the hypothesis $T_i(U)$ implies that the right arrow in the commutative square 
        \[
		\begin{tikzcd}
			\CH_i(X) \otimes_\Q \Q_\ell \arrow[r] \arrow[d] & \CH_i(U) \otimes_\Q \Q_\ell \arrow[d]\\
			H_{2i}(\ol{X}, \Q_\ell(i))^{G_{\F_q}} \arrow[r] & H_{2i}(\ol{U}, \Q_\ell(i))^{G_{\F_q}}
		\end{tikzcd}
		\]
        is surjective. Since the top arrow is also surjective, the bottom arrow is as well. We thus have a commutative diagram with exact rows
		\[
		\begin{tikzcd}
			\CH_i(Z) \otimes_\Q \Q_\ell \arrow[r] \arrow[d] & \CH_i(X) \otimes_\Q \Q_\ell \arrow[r] \arrow[d] & \CH_i(U) \otimes_\Q \Q_\ell \arrow[r, overlay] \arrow[d] & 0 \\
			H_{2i}(\ol{Z}, \Q_\ell(i))^{G_{\F_q}} \arrow[r] & H_{2i}(\ol{X}, \Q_\ell(i))^{G_{\F_q}} \arrow[r] & H_{2i}(\ol{U}, \Q_\ell(i))^{G_{\F_q}} \arrow[r, overlay] & 0.
		\end{tikzcd}
		\]
		Now the hypotheses $T_i(Z)$ and $T_i(U)$, combined with the other four-lemma, imply $T_i(X)$.
	\end{proof}
	
	\begin{lem}\label{lemBi(X)AlongExts}
		If $k \cong \F_q$, $X$ and $Y$ are finite-type smooth schemes over $k$, and $f: Y \to X$ is a $k$-morphism which factors as a composition of finite Galois covers and universal homeomorphisms, then $B^i(Y)$ implies $B^i(X)$.
	\end{lem}
	\begin{proof}
		The case where $f$ is a universal homeomorphism is clear, as then the \'etale sites (resp. the Chow groups) of $X$ and $Y$ are equivalent (resp. isomorphic). So we may assume $f$ is finite Galois. As the pullback injects $\CH^i(X)$ into $\CH^i(Y)$, the assumption $B^i(Y)$ implies that
		\begin{gather*}
			\cl: \CH^i(X) \otimes_{\Q} \Q_\ell \to H^{2i}(\ol{X}, \Q_\ell(i))^{G_k}
		\end{gather*}
		is injective. It is also surjective: the diagram
		\[
		\begin{tikzcd}
			\CH^i(Y) \otimes_\Q \Q_\ell \arrow[r, "\cl"] \arrow[d, "\Tr_{f}"] & H^{2i}(\ol{Y}, \Q_\ell(i))^{G_k} \arrow[d, "\Tr_{f}"] \\
			\CH^i(X) \otimes_\Q \Q_\ell \arrow[r, "\cl"] & H^{2i}(\ol{X}, \Q_\ell(i))^{G_k}
		\end{tikzcd}
		\]
		commutes (where $\Tr_f$ denotes the trace, and we use Lemma \ref{lemCHInjectsUnderBaseExt} to identify the codomain of the left arrow), and the top and right arrows are surjective, so the bottom arrow is as well. 
	\end{proof}
	
	\begin{lem}\label{lemBi(X)BaseChange}
		If $k \cong \F_q$, $X$ is a finite-type smooth scheme over $k$, and $L / k$ is a finite extension, then $SS^i(X_L)$ implies $SS^i(X)$, $T^i(X_L)$ implies $T^i(X)$, and $B^i(X_L)$ implies $B^i(X)$. Here $X_L$ is considered as a scheme over $L$.
	\end{lem}
	\begin{proof}
		Let $n = [L:k]$. All 1-generalized eigenvectors of $\Frob_q$ acting on $H^{2i}(\ol{X}, \Q_\ell(i))$ are also 1-generalized eigenvectors for $\Frob_q^n$, as the polynomial $x-1 \in \Z[x]$ divides $x^n - 1$. By the hypothesis $SS^i(X_L)$, 1 is a semisimple eigenvalue of $\Frob_q^n$, so all 1-generalized eigenvectors of $\Frob_q$ are fixed by $\Frob_q^n$. But this implies that they are all fixed by $\Frob_q$, as $x-1$ divides $x^n - 1$ with multiplicity one.
		\par
		That $T^i(X_L)$ implies $T^i(X)$, and $B^i(X_L)$ implies $B^i(X)$, now follows from a \enquote{transfer} argument similar to that of Lemma \ref{lemBi(X)AlongExts}.
	\end{proof}
	
	\begin{lem}\label{lemTateEquivs}
		Let $M$ be a motive over a finite field $k$ which lies in the thick and rigid tensor subcategory of $\mc{M}(k)$ generated by objects of the form $\mf{h}(A_L)$, where $A$ is an abelian variety over $k$ and $L / k$ is a finite extension. Then for $i \in \Z$, $T^i(M)$ implies $B^i(M)$.
	\end{lem}
	\begin{proof}
		This essentially follows from the arguments of \cite[Section 1]{Kahn2003}, but we take a slightly different approach, in order to be careful around some points raised in \cite[Remarks 4.11]{Jannsen2007}. By Tate twisting, we can and do assume without loss of generality that $i = 0$. As $H^\bullet(M, \Q_\ell)$ is a semisimple Galois representation (\cite[Lemme 1.9]{Kahn2003}), the homological motive $M_{\hom}$ has a direct sum decomposition 
		\begin{gather*}
			M_{\hom} \cong M'_{\hom} \oplus M''_{\hom},
		\end{gather*}
		such that the Frobenius $\Frob \in \End(M)$ acts as the identity on $H^\bullet(M'_{\hom}, \Q_\ell) = H^0(M'_{\hom}, \Q_\ell)$, and the characteristic polynomial $P(t)$ of $\Frob$ acting on $V = H^\bullet(M''_{\hom}, \Q_\ell)$ does not vanish at $t = 1$. Explicitly, $M'_{\hom}$ is cut out by the projector $T = P(\Frob) / P(1)$ on $M_{\hom}$.
		\par
		By Proposition \ref{propFdDescent}, $M$ is finite-dimensional, since there exists a finite extension $N / k$ such that $M_N$ is of abelian type. Corollary \ref{corDecompsOfFinDimMotivesLift} then implies that the decomposition of $M_{\hom}$ lifts to a direct sum decomposition
		\begin{gather*}
			M \cong M' \oplus M''
		\end{gather*}
		modulo rational equivalence. The Frobenius endomorphism commutes with all morphisms of motives and hence preserves this decomposition. As $T$ acts as zero on $V$, Kimura's nilpotence theorem (Theorem \ref{thmNilpotence}) implies that $T$ is nilpotent on $M''$. But $T$ acts as the identity on $\CH^0(M'') = \Hom(\mathbf{1}, M'')$, since the Frobenius acts as the identity on $\mathbf{1}$. Hence, $\CH^0(M'') = 0$, and $B^0(M'')$ holds. Thus $B^0(M)$ is equivalent to $B^0(M')$. 
		\par
		By the assumption $T^0(M)$, $T^0(M')$ holds. Choosing a basis of algebraic cycles for the $\ell$-adic cohomology of $M'$, we obtain a map of motives
        \begin{gather*}
            f: \mathbf{1}^{\oplus r} \to M'
        \end{gather*}
        which induces an isomorphism on $\ell$-adic realizations. By the conservativity statement of Proposition \ref{propConservativity}, $f$ is an isomorphism (since this can be checked after base change to $N$, and $M'_N$ is of abelian type). Therefore $B^0(M')$ holds.
	\end{proof}
	
	\section{The Beilinson--Bloch conjecture over function fields}
	In this section, let $k = \F_q$. Here we will prove a sufficient condition for the Beilinson--Bloch conjecture over a global function field. This will come as a corollary of a statement which applies to general fibrations over $k$ (Theorem \ref{thmBBCritArbBase}). 

    \subsection{The Leray filtration}
    Let $Y$ be a smooth, geometrically integral variety over $k$. Let $f: X \to Y$ be a smooth projective morphism, and let $i \in \Z$.
	\begin{lem}[Decomposition theorem]\label{lemDecomp}
		There is an isomorphism
		\begin{gather*}
			Rf_* \Q_\ell(i) \cong \bigoplus\limits_{r \in \Z} R^r f_* (\Q_\ell(i))[-r]
		\end{gather*}
		in the bounded constructible derived category $D^b_c(Y, \Q_\ell)$.
	\end{lem}
    
	\begin{proof}
		This is a well-known consequence of the hard Lefschetz theorem \cite[Th\'eor\`eme (4.1.1)]{Deligne1980} and the decomposability criterion of \cite[Remarque 1.10]{Deligne1968}. 
	\end{proof}
	
	Consider the Leray spectral sequence
	\begin{gather}\label{eqLeray}
		E_2^{rs} = H^r(\ol{Y}, R^s \ol{f}_* \Q_\ell(i)) \implies H^{r + s}(\ol{X}, \Q_\ell(i))
	\end{gather}
	Degeneration of \eqref{eqLeray} at $E_2$ follows from Lemma \ref{lemDecomp}, proper base change and \cite[Proposition 1.2]{Deligne1968}. The sequence yields a descending, Galois-equivariant filtration 
	\begin{gather*}
		F^\bullet H^n(\ol{X}, \Q_\ell(i))
	\end{gather*}
	on the cohomology of $\ol{X}$. We can pull back the filtration to the Chow groups of $X$ via the cycle class map, defining
	\begin{gather*}
		M^\bullet \CH^i(X) = \cl^{-1}(F^\bullet H^{2i}(\ol{X}, \Q_\ell(i))).
	\end{gather*}
	We leave the dependency of this filtration on $f$ implicit in the notation. 
    \par
    Note that the same filtration was studied in relation to the Hodge and Tate conjectures by Arapura \cite{Arapura2022}. Following loc. cit. page 4, we may equivalently define the Leray filtration on cohomology via 
    \begin{gather*}
        F^\bullet H^n(\ol{X}, \Q_\ell(i)) = \Image[H^n(\ol{Y}, \tau_{\le n - \bullet} R \ol{f}_* \Q_\ell(i)) \to H^n(\ol{Y}, R \ol{f}_* \Q_\ell(i)) = H^n(\ol{X}, \Q_\ell(i))].
    \end{gather*}
    Here $\tau$ denotes the truncation with respect to the standard $t$-structure on $D^b_c(\ol{Y}, \Q_\ell)$. From this description it is clear that for any open $U \subseteq Y$, the pullback map $H^n(\ol{X}, \Q_\ell(i)) \to H^n(\ol{X}_U, \Q_\ell(i))$ is compatible with the respective Leray filtrations associated with the maps $X \to Y$ and $X_U \to U$. Likewise, the pullback map $\CH^i(X) \to \CH^i(X_U)$ is compatible with Leray filtrations.
    \par
    Let $\eta$ be the generic point of $Y$, $\eta_\infty$ the generic point of $\ol{Y}$, and $\ol{\eta} \to \ol{Y}$ a geometric generic point over $\eta_\infty$.
	\begin{lem}\label{lemGenericFiberCohLeray}
		Fix $j \in \Z$, and let $\mc{F} = R^j \ol{f}_{*} \Z_\ell(i)$. For $U$ a nonempty open subscheme of $Y$, the pullback induces an isomorphism
		\begin{gather*}
			H^0(\ol{U}, \mc{F}|_{\ol{U}}) \cong H^j(X_{\ol{\eta}}, \Z_\ell(i))^{\Gal(\ol{\eta} / \eta_\infty)}
		\end{gather*}
		of $\Z_{\ell}[G_k]$-modules.
	\end{lem}
       \begin{proof}
           By smooth proper base change, $\mc{F}|_{\ol{U}}$ is a lisse $\Z_\ell$-sheaf on $\ol{U}$, and its stalk at $\ol{\eta}$ is identified with $H^j(X_{\ol{\eta}},\Z_\ell(i))$. By \cite[Tag 0DV4]{stacks-project}, we have 
           \begin{align*}
                H^0(\ol{U},\mc{F}|_{\ol{U}}) &=\mc{F}_{\ol{\eta}}^{\pi_1(\ol{U},\ol{\eta})} \\
                &= \mc{F}_{\ol{\eta}}^{\Gal(\ol{\eta} / \eta_\infty)},
           \end{align*}
           because the $\Gal(\ol{\eta} / \eta_\infty)$-action factors through the map $\Gal(\ol{\eta} / \eta_\infty) \to \pi_1(\ol{U},\ol{\eta})$, which is surjective \cite[Tag 0BQI]{stacks-project}. Here $\pi_1$ denotes the \'etale fundamental group.
	\end{proof}
	\begin{rem}
		The statement of Lemma \ref{lemGenericFiberCohLeray} also holds with $\Q_\ell$ in place of $\Z_\ell$, as for a finitely-generated $\Z_\ell$-module $M$ on which a group $G$ acts, 
		\begin{gather*}
			M^G \otimes_{\Z_\ell} \Q_\ell \cong (M \otimes_{\Z_\ell} \Q_\ell)^G
		\end{gather*}
		via the natural map.
	\end{rem}

    We now study the interaction of the Leray filtration with weights. For this, we will need an elementary linear-algebraic lemma. 

    \begin{lem}\label{lemFilteredVS}
        Let $V$ and $W$ be vector spaces over a field, equipped with finite descending filtrations $(\Fil^\bullet V)_{\bullet \ge 0}$ and $(\Fil^\bullet W)_{\bullet \ge 0}$, such that $\Fil^0 V = V$ and $\Fil^0 W = W$. If $g: V \to W$ is a surjective map of filtered vector spaces, and for some $n \ge 0$, the map
        \begin{gather*}
            V / \Fil^n V \to W / \Fil^n W
        \end{gather*}
        induced by $g$ is injective, then the maps
        \begin{gather*}
            \Fil^n V \to \Fil^n W,\\
            \gr^n V \to \gr^n W
        \end{gather*}
        induced by $g$ are surjective.
    \end{lem}
    \begin{proof}
        This is immediate from the four-lemma applied to the commutative diagram with exact rows
        \[
            \begin{tikzcd}
                0 \arrow[r] & \Fil^n V \arrow[r] \arrow[d] & V\arrow[r] \arrow[d] & V / \Fil^n V \arrow[d] \\
                0 \arrow[r] & \Fil^n W \arrow[r] & W \arrow[r] & W / \Fil^n W.
            \end{tikzcd}
        \]
    \end{proof}

    \begin{lem}\label{lemResIsoFrobSim1}
        Let $Z \subseteq Y$ be a closed subvariety of codimension $c \ge 1$, and let $U = Y - Z$. Then for all $0 \le r < 2c$, the pullback
        \begin{gather}\label{eqPullbackGrLeray}
            H^r(\ol{Y}, R^{2i-r} \ol{f}_* \Q_\ell(i))^{\Frob_q \sim 1} \to H^r(\ol{U}, R^{2i-r} \ol{f}_{U, *} \Q_\ell(i))^{\Frob_q \sim 1}
        \end{gather}
        is an isomorphism. 
    \end{lem}
    \begin{proof}
        Let $j: U \hookrightarrow Y$ be the inclusion. By semi-purity \cite[Lemma 23.1]{MilneLEC}\footnote{Note that \cite[Lemma 23.1]{MilneLEC} is stated for a constant sheaf, but the argument extends to the lisse case, via \cite[Theorem 16.1]{MilneLEC} and \cite[Remark III.1.26]{Milne1980}.}, the map \eqref{eqPullbackGrLeray}
        is an isomorphism for $0 \le r < 2c - 1$, and injective for $r = 2c-1$. Lemma \ref{lemDecomp} implies that \eqref{eqPullbackGrLeray} is the map induced by
        \begin{gather}\label{eqFrobSim1Pullback}
            j^*: H^{2i}(\ol{X}, \Q_\ell(i))^{\Frob_q \sim 1} \to H^{2i}(\ol{X}_U, \Q_\ell(i))^{\Frob_q \sim 1}
        \end{gather}
        on the graded components of the respective Leray filtrations. By Proposition \ref{propHomologyProps}(ii) and (v), \eqref{eqFrobSim1Pullback} is surjective. Now Lemma \ref{lemFilteredVS} shows that when $r = 2c-1$, \eqref{eqPullbackGrLeray} is surjective, and hence an isomorphism.
    \end{proof}
    \begin{thm}\label{thmBBCritArbBase}
         Let $Z \subseteq Y$ be a closed subvariety of codimension $c \ge 1$, and let $U = Y - Z$. Assume that $B^i(X)$ holds. Then for all $0 \le r < 2c$, the cycle class map
        \begin{align*}
            \gr^r_M \CH^i(X_U) \otimes_{\Q} \Q_\ell \to H^r(\ol{U}, R^{2i-r} \ol{f}_{U, *} \Q_\ell(i))^{\Frob_q \sim 1}
        \end{align*}
        is an isomorphism. Here $\gr^r_M$ denotes the $r$-th associated graded component of the Leray filtration induced by $X_U \to U$. In particular, the kernel of the pullback to the generic fiber $\CH^i(X) \to \CH^i(X_\eta)$ is contained in $M^2 \CH^i(X)$. 
    \end{thm}
    \begin{proof}
        By the definition of the Leray filtration, the assumption $B^i(X)$, and Lemma \ref{lemDecomp}, the cycle class map is an isomorphism
        \begin{gather*}
            \gr^r_M \CH^i(X) \otimes_{\Q} \Q_\ell \to H^r(\ol{Y}, R^{2i-r} \ol{f}_* \Q_\ell(i))^{\Frob_q \sim 1}.
        \end{gather*}
        The diagram 
        \begin{equation}\label{eqChResSquare}
            \begin{tikzcd}
                \gr^r_M \CH^i(X) \otimes_\Q \Q_\ell \arrow[r] \arrow[d] & H^r(\ol{Y}, R^{2i-r} \ol{f}_* \Q_\ell(i))^{\Frob_q \sim 1} \arrow[d] \\
                \gr^r_M \CH^i(X_U) \otimes_\Q \Q_\ell \arrow[r] & H^r(\ol{U}, R^{2i-r} \ol{f}_{U, *} \Q_\ell(i))^{\Frob_q \sim 1}
            \end{tikzcd}
        \end{equation}
        commutes, the top arrow is an isomorphism, and the right arrow is also an isomorphism (by Lemma \ref{lemResIsoFrobSim1}). We now argue by induction on $r$ that the left arrow is surjective, and hence all the arrows in the square are isomorphisms. For $r = 0$, this claim is immediate from surjectivity of $\CH^i(X) \to \CH^i(X_U)$. For $0 < r \le 2c - 1$, the surjectivity of the left arrow follows from the inductive hypothesis and Lemma \ref{lemFilteredVS}.
        \par
        Because $c \ge 1$, all the arrows in the diagram \eqref{eqChResSquare} are isomorphisms for $r \in \{0, 1\}$. Taking the colimit of the diagram over nonempty open subschemes $U$ properly contained in $Y$, we find that the kernel of $\CH^i(X) \to \CH^i(X_\eta)$ is contained in $M^2 \CH^i(X)$.
    \end{proof}
    \begin{rem}
        The proof of Theorem \ref{thmBBCritArbBase} does not require the full strength of the hypothesis $B^i(X)$, but rather just that
        \begin{gather*}
            (\CH^i(X) / M^{2c} \CH^i(X)) \otimes_\Q \Q_\ell \to [H^{2i}(\ol{X}, \Q_\ell(i)) / F^{2c} H^{2i}(\ol{X}, \Q_\ell(i))]^{\Frob_q \sim 1}
        \end{gather*}
        is an isomorphism.
    \end{rem}

    \begin{prop}\label{propTateCriterion}
	    If the Tate conjecture $T^i(X)$ holds in codimension $i$ for $X$, then so does $T^i(X_\eta)$.
	\end{prop}
	\begin{proof}
		This follows from the decomposition theorem (Lemma \ref{lemDecomp}) and Lemma \ref{lemGenericFiberCohLeray}, applied to $R^{2i} \ol{f}_* \Q_\ell(i)$.
	\end{proof}

    \begin{rem}
        Proposition \ref{propTateCriterion} is not new; see e.g. \cite{Andre1996} and \cite{Ambrosi2018}. 
    \end{rem}

    \begin{rem}\label{remMotivicSurjGenericFiber}
		In the situation of Proposition \ref{propTateCriterion}, work of Jannsen quickly generalizes the proposition to arbitrary motivic cohomology groups of $X$, as follows. Fix integers $r$ and $s$, and consider the map
		\begin{gather*}
			g_{r, s}: H^r_{\mc{M}}(X, \Q(s)) \otimes_\Q \Q_\ell \to H^r(\ol{X}, \Q_\ell(s))^{G_{k}}.
		\end{gather*}
		from the motivic cohomology to the \'etale cohomology of $X$, as in \cite[Conjecture 12.4]{Jannsen1990}. Then if $g_{r, s}$ is surjective, the analogous map on the cohomology of the generic fiber
		\begin{gather*}
			H^r_{\mc{M}}(X_{k(Y)}, \Q(s)) \otimes_\Q \Q_\ell \to H^r(X_{\ol{k(Y)}}, \Q_\ell(s))^{G_{k(Y)}}
		\end{gather*}
		is also surjective. This claim follows from the argument of \cite[Theorem 12.16 b)]{Jannsen1990}. Note that the statement in loc. cit. assumes an additional 1-semisimplicity hypothesis, but the decomposition theorem (Lemma \ref{lemDecomp}) actually makes this unnecessary under our assumption that $f$ is smooth and projective. Jannsen also assumes the base $Y$ is a curve, but the needed part of his argument goes through for higher-dimensional $Y$. Proposition \ref{propTateCriterion} is the case  $(r, s) = (2i, i)$ of the more general statement.
	\end{rem}

    \subsection{The Beilinson--Bloch conjecture}
    We now specialize the above discussion to the case where the base of the fibration is a curve. First, we introduce some notation. For a constructible $\Q_\ell$-sheaf $\mc{F}$ on a finite-type $k$-scheme $X$, we define the zeta and $L$-functions of $\mc{F}$ via
	\begin{align*}
		Z(X, \mc{F}, t) &= \prod\limits_{x \in X^0} \chi(\mc{F}_{\ol{x}}, t^{[k(x):k]})^{-1} \in \Q_\ell[[t]], \\
		L(X, \mc{F}, s) &= Z(X, \mc{F}, q^{-s}).
	\end{align*}
	Here $X^0$ is the set of closed points of $X$, and $\ol{x} = \Spec \ol{k(x)}$ for a choice of algebraic closure $\ol{k(x)} / k(x)$. In all examples we consider, the function $Z(X, \mc{F}, t)$ will lie in $\Q(t) \subseteq \Q_\ell((t))$ (by a combination of the Grothendieck--Lefschetz trace formula and the Weil conjectures), and hence can be evaluated at a complex argument.
	\par
	For a global field $K$ and continuous $G_K$-representation on a finite-dimensional $\Q_\ell$-vector space $V$, we define the Hasse--Weil $L$-function via
	\begin{gather*}
		L(V, s) = \prod\limits_{w} \chi((V|_{G_w})^{I_w}, |\kappa_w|^{-s})^{-1}, 
	\end{gather*}
	where the product is over the set of equivalence classes of non-archimedean valuations of $K$, $K_w$ is the completion of $K$ at $w$, $G_w$ is the decomposition group of a choice of place of $\ol{K}$ over $w$, $I_w \subseteq G_w$ is the inertia subgroup, and $\kappa_w$ is the residue field of $w$.
    \par
	With this notation in place, recall the statement of the Beilinson--Bloch conjecture.
	\begin{conj}[$BB^i(V)$]\label{conjBB}
		Let $K$ be a global field, $V$ a smooth projective variety over $K$, and $i \in \Z$. Then
		\begin{equation*}
			\dim_\Q \CH^i(V)_0 = \ord_{s = i} L(H^{2i-1}(\ol{V}, \Q_\ell), s).
		\end{equation*}
	\end{conj}
	\par
	We denote the truth value of Conjecture \ref{conjBB} for a given $V$ and $i$ by $BB^i(V)$. The conjecture also makes sense for a motive $M \in \mc{M}(K)$, and we similarly write $BB^i(M)$ for the truth value of the conjecture for $M$ and $i$.
	\par
    The connection of Theorem \ref{thmBBCritArbBase} to the Beilinson--Bloch conjecture arises by combining the Grothendieck--Lefschetz trace formula with some weight-related results from Weil II \cite{Deligne1980}. Throughout this subsection, let $C$ denote a smooth, geometrically integral curve over $k$, with function field $K$. 
    \begin{prop}\label{propSpectralL}
        Let $f: X \to C$ be a smooth projective morphism, with $X$ of pure dimension $d$. For $1 \le i \le d - 1$, we have that
        \begin{gather*}
            \ord\limits_{s=i} L(H^{2i-1}(X_{\ol{K}}, \Q_\ell), s) = \ord\limits_{t=1} \chi(H^1(\ol{C}, R^{2i-1} \ol{f}_* \Q_\ell(i)), t).
        \end{gather*}
    \end{prop}
    \begin{proof}
		By Lemma \ref{lemResIsoFrobSim1}, for any nonempty open $U \subseteq C$, the restriction map 
        \begin{gather*}
            H^1(\ol{C}, R^{2i-1} \ol{f}_* \Q_\ell(i))^{\Frob_q \sim 1} \to H^1(\ol{U}, R^{2i-1} \ol{f}_{U, *} \Q_\ell(i))^{\Frob_q \sim 1}
        \end{gather*}
        is an isomorphism. The $\Q_\ell$-dimension of this space is precisely $\ord\limits_{t=1} \chi(H^1(\ol{U}, R^{2i-1} \ol{f}_{U, *} \Q_\ell(i)), t)$, so we can and do assume without loss of generality that $C$ is affine.
        \par
        By \cite[Lemme (3.6.2)]{Deligne1980}, each of the factors in the Euler product defining $L(H^{2i-1}(X_{\ol{K}}, \Q_\ell), s)$ is holomorphic (and evidently nonzero) at $s = i$. Therefore
		\begin{gather*}
			\ord_{s=i} L(H^{2i-1}(X_{\ol{K}}, \Q_\ell), s) = \ord_{s=i} L(\ol{C}, R^{2i-1}\ol{f}_{*} \Q_\ell, s),
		\end{gather*}
		as the latter function is obtained by omitting finitely many factors from the Euler product. 
		\par
		Let $\mc{F} = R^{2i-1}\ol{f}_{*} \Q_\ell$. By the Grothendieck--Lefschetz trace formula \cite[Theorem VI.13.3]{Milne1980}, 
		\begin{gather}\label{eqGrotLef}
			L(\ol{C}, \mc{F}, s) = \frac{\chi(H^1_c(\ol{C}, \mc{F}), q^{-s})}{\chi(H^0_c(\ol{C}, \mc{F}), q^{-s})\chi(H^2_c(\ol{C}, \mc{F}), q^{-s})}.
		\end{gather}
		Artin vanishing \cite[Tag 0F0W]{stacks-project} and Poincar\'e duality \cite[Corollary II.7.1]{KW2001} imply that $H^0_c(\ol{C}, \mc{F}) = 0$. By \cite[Corollaire (1.4.3)]{Deligne1980} and the proof of the Weil conjectures in loc. cit., 
		\begin{gather*}
			\chi(H^2_c(\ol{C}, \mc{F}), q^{-i}) \neq 0.
		\end{gather*}
		Thus
		\begin{gather*}
			\ord_{s=i} L(H^{2i-1}(X_{\ol{K}}, \Q_\ell), s) = \ord_{s=i} \chi(H^1_c(\ol{C}, \mc{F}), q^{-s}).
		\end{gather*}
		\par
		Let $\mc{G} = R^{2d - 2i - 1} \ol{f}_{*} \Q_\ell(d-2i)$. Since the fiber dimension of $f$ is $d-1$, by \cite[Th\'eor\`eme (4.1.1)]{Deligne1980} there is a hard Lefschetz isomorphism 
		\begin{gather*}
			\lambda^{d-2i}: \mc{F} \xrightarrow{\sim} \mc{G},
		\end{gather*}
		given by repeatedly cupping with the class of a relatively ample line bundle $\lambda \in H^0(\ol{C}, R^2 \ol{f}_{*} \Q_\ell(1))$ (or given by the inverse of the map just described, if $d - 2i < 0$). This yields a Galois-equivariant isomorphism $H^1_c(\ol{C}, \mc{F}) \cong H^1_c(\ol{C}, \mc{G})$. By Poincar\'e duality, $\mc{G}(i)^\vee \cong \mc{F}(i-1)$, and $H^1_c(\ol{C}, \mc{G}(i)) \cong H^1(\ol{C}, \mc{F}(i))^\vee$, so 
		\begin{gather*}
			\ord_{s=i} L(H^{2i-1}(X_{\ol{K}}, \Q_\ell), s) = \ord_{t=1} \chi(H^1(\ol{C}, \mc{F}(i)), t),
		\end{gather*}
        as desired.
	\end{proof}
    \begin{cor}\label{corBBCriterion}
        Assume that $C$ is affine. Let $f: X \to C$ be a smooth projective morphism, with $X$ of pure dimension $d$. For $1 \le i \le d - 1$, if $B^i(X)$ holds, then $BB^i(X_K)$ holds.
    \end{cor}
    \begin{proof}
        Artin vanishing \cite[Tag 0F0W]{stacks-project} implies that $H^2(\ol{C}, R^{2i-2} \ol{f}_* \Q_\ell(i)) = 0$. Equivalently, 
        \begin{gather*}
            F^2 H^{2i}(\ol{X}, \Q_\ell(i)) = 0.
        \end{gather*}
        The hypothesis $B^i(X)$ likewise yields
        \begin{gather*}
            M^2 \CH^i(X) = 0.
        \end{gather*}
        Lemma \ref{lemGenericFiberCohLeray} and Theorem \ref{thmBBCritArbBase} imply that the pullback map
        \begin{gather*}
            \gr_0^M \CH^i(X) \to \CH^i_{\hom}(X_K)
        \end{gather*}
        is an isomorphism (where $\CH^i_{\hom}(X_K)$ is the Chow group modulo homological equivalence). From Lemma \ref{lemFilteredVS} and Theorem \ref{thmBBCritArbBase} we then find that
        \begin{gather*}
            M^1 \CH^i(X) \to \CH^i(X_K)_0
        \end{gather*}
        is an isomorphism, where we apply the lemma to the two-step filtration $\CH^i(X_\eta)_0 \subseteq \CH^i(X_\eta)$ on $\CH^i(X_\eta)$. As $B^i(X)$ implies that
        \begin{gather*}
            M^1 \CH^i(X) \otimes_\Q \Q_\ell \cong H^1(\ol{C}, R^{2i-1} \ol{f}_{*} \Q_\ell(i))^{\Frob_q \sim 1},
        \end{gather*}
        from Proposition \ref{propSpectralL} we now infer $BB^i(X_K)$.
    \end{proof}
    \begin{rem}\label{remAJ}
		Fix notation and hypotheses as in Corollary \ref{corBBCriterion}. Using Theorem \ref{thmBBCritArbBase} and \cite[Remark 12.17 b)]{Jannsen1990}, one may show that $B^i(X)$ implies that all the arrows in the commutative diagram 
		\[
			\begin{tikzcd}
				M^1 \CH^i(X) \otimes_{\Q} \Q_\ell \arrow[r] \arrow[d] & H^1(\ol{C}, R^{2i-1} \ol{f}_{*} \Q_\ell(i))^{\Frob_q \sim 1} \arrow[d] \\
				\CH^i(X_K)_{0} \otimes_{\Q} \Q_\ell \arrow[r] & H^1_{\cont}(G_K, H^{2i-1}(X_{\ol{K}}, \Q_\ell(i)))
			\end{tikzcd}
		\]
		are isomorphisms. The bottom arrow is the Abel--Jacobi map from $\CH^i(X_K)_0 \otimes \Q_\ell$ into continuous group cohomology. 
	\end{rem}
	
	We now relax the smoothness and affineness assumptions in Corollary \ref{corBBCriterion}, then refine it in the case $i = 1$. 
	
	\begin{cor}\label{corClosedPointsBBCriterion}
		Let $f: X \to C$ be a finite-type, flat, separated morphism such that $X_K$ is a smooth projective variety over $K$. Let $Z \subsetneq C$ be a finite, closed subscheme, with $U = C - Z$, such that $U$ is affine and $f_U$ is smooth and projective. Let $d = \dim X$, and $0 \le i \le d - 1$. If there exists a dense open $V \subseteq C$ such that $B_{d-i}(X_V)$ holds, and $T_{d-i}(X_{V \cap Z})$ and $SS_{d-i}(X_{V \cap Z})$ hold, then $BB^i(X_K)$ and $T^i(X_K)$ hold. 
	\end{cor}
	\begin{proof}
		By Lemma \ref{lemTateForOpen}, the hypotheses imply that $B_{d-i}(X_U \cap X_V) = B^i(X_{U \cap V})$ holds, so that Proposition \ref{propTateCriterion} and Corollary \ref{corBBCriterion} apply.
	\end{proof}
    \begin{rem}
        Via an identical argument, one can generalize Theorem \ref{thmBBCritArbBase} to non-smooth morphisms. To avoid overloading the exposition, we omit the details. 
    \end{rem}
	
	\begin{thm}\label{thmBB1AndT1ImplyB1}
		Assume that $C$ is projective. Let $f: X \to C$ be a flat projective morphism, with $X$ smooth over $k$, and $X_K$ a smooth variety over $K$. Then $T^1(X)$ holds iff $BB^1(X_K)$ and $T^1(X_K)$ hold.
	\end{thm}
	\begin{proof}
		Let $d = \dim X$. Lemma \ref{lemAdEqRelsOnDivisors}, together with \cite[Theorem 2.9]{Tate1994}, shows that $T^1(X)$ implies $B^1(X) = B_{d-1}(X)$. For any closed point $y \in C$, \cite[Proposition (2.4)]{Laumon1976} implies that $B_{d-1}(X_y)$ holds, since $\dim X_y = d-1$. Thus Corollary \ref{corClosedPointsBBCriterion} applies to show $BB^1(X_K)$ and $T^1(X_K)$.
		\par
		Now we prove the converse. Let $U \subset C$ be an affine open such that $f_U$ is smooth. Assuming $BB^1(X_K)$, Proposition \ref{propSpectralL} shows that
		\begin{gather*}
			\dim_\Q \CH^1(X_K)_0 = \ord\limits_{t=1} \chi(H^1(\ol{U}, R^1 f_{\ol{U}, *} \Q_\ell(1)), t).
		\end{gather*}
		The assumption $T^1(X_K)$ and Lemma \ref{lemGenericFiberCohLeray} then imply that Chow group modulo homological equivalence $\CH^1_{\hom}(X_K)$ satisfies
        \begin{gather*}
            \dim_\Q \CH^1_{\hom}(X_K) \ge \dim_{\Q_\ell} H^0(\ol{U}, R^2 f_{\ol{U}, *} \Q_\ell(1))^{G_{k}}.
        \end{gather*}
        Then the decomposition theorem (Lemma \ref{lemDecomp}) applied to $f_U$ yields
		\begin{gather}\label{eqGenericChEqualsFrobMult}
			\dim_\Q \CH^1(X_K) \ge \dim_{\Q_{\ell}} H^2(\ol{X}_U, \Q_\ell(1))^{G_{k}}.
		\end{gather} 
		Let $Z = C - U$, and consider the commutative diagram with exact rows
		\[
		\begin{tikzcd}
			\CH_{d-1}(X_Z) \otimes_\Q \Q_\ell \arrow[r] \arrow[d] & \CH_{d-1}(X) \otimes_\Q \Q_\ell \arrow[r] \arrow[d] & \CH_{d-1}(X_U) \otimes_\Q \Q_\ell \arrow[r, overlay] \arrow[d] & 0 \\
			H_{2d-2}(\ol{X}_Z, \Q_\ell(d-1))^{\Frob_q \sim 1} \arrow[r] & H_{2d-2}(\ol{X}, \Q_\ell(d-1))^{\Frob_q \sim 1} \arrow[r] & H_{2d-2}(\ol{X}_U, \Q_\ell(d-1))^{\Frob_q \sim 1} \arrow[r, overlay] & 0.
		\end{tikzcd}
		\]
		The left vertical arrow is an isomorphism by \cite[Proposition (2.4)]{Laumon1976}, and the middle vertical arrow is injective by Lemma \ref{lemAdEqRelsOnDivisors}, so the right vertical arrow is injective by the four-lemma. As $\CH^1(X_U)$ surjects onto $\CH^1(X_K)$, we conclude from \eqref{eqGenericChEqualsFrobMult} and Poincar\'e duality that $T^1(X_U)$ holds. Lemma \ref{lemTateForOpen}(ii) now implies that $T^1(X)$ holds.
	\end{proof}
	
	When $X_K$ is geometrically integral, Theorem \ref{thmBB1AndT1ImplyB1} is implied by \cite[Theorem 1.1]{Geisser2021}, and is in fact equivalent to it (modulo a caveat), as we now explain.
	
	\begin{thm}\label{thmATGeneralization}
			Let $f: X \to C$ be a flat projective morphism, with $X$ a smooth variety over $k$, and $X_K$ smooth and geometrically integral over $K$. The following are equivalent:
		\begin{enumerate}[(i)]
			\item the Tate conjecture holds for divisors on $X$;
			\item the Tate conjecture holds for divisors on $X_K$, and the Birch and Swinnerton-Dyer conjecture holds for $\Alb(X_K)$;
			\item the Brauer group of $X$ is finite;
			\item the Tate conjecture holds for divisors on $X_K$, and the Tate--Shafarevich group of $\Alb(X_K)$ is finite.
		\end{enumerate}
	\end{thm}
	\begin{proof}
		The equivalence of (i) and (ii) follows from Theorem \ref{thmBB1AndT1ImplyB1}. Using the Kummer sequence and the Hochschild--Serre spectral sequence, one checks that (iii) implies (i), and (i) implies (iii) by \cite[Theorem 0.4(b)]{Milne1986}. Now (ii) is equivalent to (iv) by \cite[Theorem on page 541]{KT2003}.
	\end{proof}
	The equivalence of the statements (iii) and (iv) of Theorem \ref{thmATGeneralization} is \cite[Theorem 1.1]{Geisser2021} (with the caveat that Geisser allows $X$ to be proper over $C$, rather than projective). 
	\par
	We now want to apply Proposition \ref{propTateCriterion} and Corollary \ref{corBBCriterion} to constant varieties of abelian type over function fields.
    \begin{prop}\label{propBBAndTateAfterBaseChange}
		Assume that $C$ is affine. Let $f: X \to C$ be a smooth projective morphism, and $N / K$ be a finite field extension. Let $D$ be a smooth curve over the constant field $L$ of $N$, with function field $k(D) \cong N$, and $g: D \to C$ a map inducing the extension $N / K$ on function fields, such that $g$ factors as a composition of universal homeomorphisms and finite Galois covers. If $B^i(X_D)$ holds (where $X_D$ is considered as a scheme over $L$), then $B^i(X)$, $BB^i(X_K)$, and $T^i(X_K)$ hold.
	\end{prop}
	\begin{proof}
		If $g$ is a universal homeomorphism, then $L = k$ and $B^i(X_D)$ is equivalent to $B^i(X)$, so we may assume that $N / K$ and $g$ are finite Galois. We consider the diagram 
		\[
		\begin{tikzcd}
			 X_D / k \arrow[r] \arrow[d] & X \arrow[d] \\
			D / k \arrow[r] & C
		\end{tikzcd}
		\]
		of $k$-schemes and base change it to $L$, where the notation $D / k$ means $D$ considered as a $k$-scheme rather than an $L$-scheme, and likewise for $X_D / k$. As $(X_D / k)_L$ is a disjoint union of copies of $X_D$, $B^i((X_D / k)_L)$ holds, and hence $B^i(X_L)$ holds by Lemma \ref{lemBi(X)AlongExts}. Thus $B^i(X)$ holds by Lemma \ref{lemBi(X)BaseChange}, and $T^i(X_K)$ and $BB^i(X_K)$ now follow from Proposition \ref{propTateCriterion} and Corollary \ref{corBBCriterion}.
	\end{proof}

	\begin{prop}\label{propBBForAbMots}
		Assume that $C$ is projective. Let $X$ be a smooth projective variety over $k$, and assume that there exist finite field extensions $L / k$ and $N / k$ such that the motive of $X_L$ in $\mc{M}(L)$ is of abelian type, $C$ has an $N$-point, and $T^i(X_L \times_L C_L)$ and $T^{i-1}(X_N)$ hold (where $X_L \times_L C_L$ and $X_N$ are considered as schemes over $L$ and $N$, respectively). Then for all nonempty open $U \subseteq C$, there exists a nonempty open $U' \subseteq U$ such that $B^i(X \times_{k} U)$ holds. Also, $BB^i(X_K)$ and $T^i(X_K)$ hold. 
	\end{prop}
	\begin{proof}
		The base change of $X$ to the compositum of $L$ and $N$ has motive of abelian type, and hence has semisimple $\ell$-adic cohomology \cite[Lemme 1.9]{Kahn2003}. Lemma \ref{lemBi(X)BaseChange} thus implies that $SS^{i-1}(X_N)$ and $T^{i-1}(X_N)$ hold, where $X_N$ is considered as a scheme over $k$. Likewise, Lemmas \ref{lemBi(X)BaseChange} and \ref{lemTateEquivs} imply that $B^i(X \times_{k} C)$ holds. Now fix an $N$-point $c \in C$ with complement $W$, and use Lemma \ref{lemTateForOpen} to infer $B^i(X \times_{k} W)$. From Theorem \ref{thmBBCritArbBase} one obtains $B^i(X \times_{k} U')$ for all nonempty open $U' \subseteq W$. Similarly, from Corollary \ref{corClosedPointsBBCriterion} one infers $BB^i(X_K)$ and $T^i(X_K)$.
	\end{proof}
	
	\begin{rem}
		Using the Tate conjecture for divisors on abelian varieties over finite fields \cite{Tate1966}, Proposition \ref{propBBForAbMots} recovers the rank part of the Birch and Swinnerton-Dyer conjecture for constant abelian varieties over function fields \cite[Theorem 3]{Milne1968}.
	\end{rem}

    \begin{cor}\label{corBBForAbMotsBaseChange}
        Assume that $C$ is projective. Let $K' \subseteq K$ be a subfield of finite degree, and assume the extension $K / K'$ factors as a tower of purely inseparable extensions and Galois extensions. Let $V$ be a smooth projective variety over $K'$, and suppose that there exists a smooth projective variety $X$ over $k$ such that $V_K \cong X \times_{k} K$, and $X$ satisfies the hypotheses of Proposition \ref{propBBForAbMots} for some fixed $i \in \Z$. Then $BB^i(V)$ and $T^i(V)$ hold.
    \end{cor}
    \begin{proof}
        Let $C'$ be the smooth projective curve with function field $K'$. By spreading out $V$ to a fibration over a sufficiently small open subscheme of $C'$, the claim follows from Propositions \ref{propBBAndTateAfterBaseChange} and \ref{propBBForAbMots}.
    \end{proof}
	
	Proposition \ref{propBBForAbMots} can be used to derive the following result of Kahn.
	\begin{cor}[{\cite[Corollary 1]{Kahn2021}}]\label{corFinAlbKer}
		For $S$ a smooth, projective, geometrically integral surface over $k$ such that the motive of $\ol{S}$ is a summand of the motive of an abelian variety over $\ol{k}$, the kernel $T(S_K)$ of the Albanese map $\CH^2(S_K)_{0, \Z} \to \Alb(S_K)(K)$ is finite.
	\end{cor}
	\begin{proof}
		By the Tate conjecture for divisors on abelian varieties, $T^1(S)$ holds. Similarly, $T^1(S \times_{k} C)$ holds, and $T^2(S \times_{k} C)$ follows from the hard Lefschetz theorem. Proposition \ref{propBBForAbMots} now implies that $BB^1(S_K)$ and $BB^2(S_K)$ hold. By the argument of \cite[Lemma 5.1]{Beilinson1987}, $T(S_K)$ is torsion. As in Kahn's original proof, we now note that by \cite[Corollary 84]{Kahn2005} and the localization exact sequence, $T(S_K)$ is finite.
	\end{proof}

    \begin{rem}\label{remJannsenContrastDetailed}
        In Corollaries \ref{corFinAlbKer} and \ref{corBBForEg}, we use cases of the Tate conjecture, together with Proposition \ref{propBBForAbMots}, to prove the Beilinson--Bloch conjecture for certain varieties $V$ over global function fields $K$. For each $V$ covered by these corollaries, the relevant cases of the Tate conjecture can also be combined with \cite[Corollary 4.10]{Jannsen2007} to show that $V$ satisfies the hypotheses of loc. cit. Theorem 5.3. In more detail, there exists a finite extension $N / K$ with the following properties: 
        \begin{enumerate}[(i)]
            \item there exists a smooth projective variety $X$ over the constant field $L$ of $N$ such that $V_N$ is isomorphic to $X \times_{L} N$;
            \item $X$ has motive of abelian type;
            \item for all finite extensions $E / L$, $X_E$ satisfies the full Tate conjecture;
            \item the product $X \times_{L} D$ satisfies the full Tate conjecture. Here $D$ is the smooth projective curve over $k$ with function field $N$.
        \end{enumerate}
        Then one argues as in the proof of \cite[Theorem 5.1]{Jannsen2007} to show that Jannsen's Theorem 5.3 applies to $V$.\footnote{The reason why we do not cite the statement of  \cite[Theorem 5.1]{Jannsen2007} directly here is that it seems to implicitly require an additional hypothesis. Specifically, with hypotheses and notation as in that theorem, to obtain the theorem's conclusion one must also assume that for all finite extensions $F' / F$, the assumptions of \cite[Corollary 4.10]{Jannsen2007} hold for $Y_{F'}$.} This theorem has consequences for all motivic cohomology groups of $V$: it implies that $H^i_{\mc{M}}(V, \Q(j)) = 0$ for $i - 2j \notin \{0, 1\}$, and computes the remaining motivic cohomology groups in terms of certain simpler invariants related to $\ell$-adic cohomology. In particular, it implies that the Abel--Jacobi map of Remark \ref{remAJ} is an isomorphism for $V$ and for all $i$. Jannsen's Theorem 5.3 also implies that Murre's conjecture holds for $V$ \cite[Conjecture 4.1]{Jannsen2007}.
        \par
        Since \cite[Theorem 5.3]{Jannsen2007} is not stated in terms of $L$-functions, deducing the Beilinson--Bloch conjecture for $V$ from Jannsen's result still requires some additional argument, which is ultimately supplied by Corollary \ref{corBBCriterion}. In a follow-up paper \cite{Broe2025}, we use Corollary \ref{corBBCriterion} to prove cases of the Beilinson--Bloch conjecture for varieties which we do not know to verify the conditions of \cite[Theorem 5.3]{Jannsen2007}. In particular, we do not control all motivic cohomology groups of these varieties, but just some of their Chow groups. See \cite[Section 6]{Broe2025}, and particularly loc. cit. Theorems 6.4 and 6.5. With the exception of the $d = 7$ case of the latter theorem, for the varieties covered by these results, it may be possible to use Lemma \ref{lemTateForOpen} to check the hypotheses of \cite[Theorem 12.16 a,b,c]{Jannsen1990} for the relevant Chow groups. We do not pursue this, as it would yield a weaker result than what we actually prove, again because Jannsen does not refer to $L$-functions.
    \end{rem}
	
	\section{CM elliptic curves}\label{secArithOfCMCurves}
	
	The remainder of the paper is dedicated to studying the Chow groups of powers of CM elliptic curves. In this section we recall some basic properties of such curves. Let $k$ be a field.
	\begin{lem}[{\cite[Theorem 1.1]{Oort1973}}]\label{lemConstAfterBaseChange}
		For any CM elliptic curve $E$ over $k$, there exists a CM elliptic curve $E'$ over a finite extension $L$ of the prime field of $k$, and a common field extension $N$ of both $k$ and $L$, with $N / k$ finite, such that $E_N$ and $E'_N$ are isogenous over $N$.
	\end{lem}
	Indeed, one can even require that $E_N$ and $E'_N$ are isomorphic over $N$, but we will only need the weaker statement.
	\par
	Now let $k \cong \F_q$ be a finite field, with $q = p^e$ for $p$ prime. Let $E$ be an ordinary elliptic curve over $k$, and write
    \begin{gather*}
        \chi(H^1(\ol{E}, \Q_\ell), t) = (1-\alpha t)(1-\ol{\alpha} t),
    \end{gather*}
    so $\alpha$ and $\ol{\alpha}$ are algebraic integers. Let $F = \Q(\alpha) \cong \End^0_k(E)$.
	
	\begin{lem}\label{lemFrobEigenvalue}
		The following are true:
		\begin{enumerate}[(i)]
			\item The ideal $p \mc{O}_F \subset \mc{O}_F$ splits as $p \mc{O}_F = \mf{p}_1 \mf{p}_2$ for two distinct prime ideals $\mf{p}_1, \mf{p}_2$ of $\mc{O}_F$, such that $\alpha \mc{O}_F = \mf{p}_1^e$ and $\ol{\alpha} \mc{O}_F = \mf{p}_2^e$.
			\item For $r > 0$ and any Weil $q$-integer $\beta$ of weight $w$ with $0 \le w < r$, $\alpha^r \beta$ is not a power of $q$. In particular, for all $s \ge 2$, $q^s / \alpha^{2s-1}$ is not an algebraic integer.
			\item For any Weil $q$-integer of weight one $\beta \notin \{\alpha, \ol{\alpha}\}$, and any $r, s \in \Z$, $\alpha^r \beta \neq q^s$.
		\end{enumerate}
	\end{lem}
	\begin{proof}
		Since $\alpha \ol{\alpha} = q$, each prime ideal of $\mc{O}_F$ dividing $\alpha \mc{O}_F$ must lie over the rational prime $p$. For any prime $\mf{p} \in \Spec \mc{O}_F$ lying over $p$, $\mf{p}$ cannot contain both $\alpha$ and $\ol{\alpha}$, or else $\Tr(\alpha) = \alpha + \ol{\alpha} \in \mf{p} \cap \Z = p \Z$, which would imply that $E$ was supersingular \cite[Exercise V.5.10]{Silverman2009}. So $\alpha \mc{O}_F$ and $\ol{\alpha} \mc{O}_F$ have no prime divisors in common, which implies (i). 
		\par
		For (ii), note that (i) implies that for $s < r$, $q^s / \alpha^r$ is not an algebraic integer. 
		\par
		For (iii), suppose that $\alpha^r \beta = q^s$, so that
		\begin{gather*}
			\beta = q^s / \alpha^r = \alpha^{s-r} \ol{\alpha}^s.
		\end{gather*}
		As $\beta$ is an algebraic integer and $\alpha \mc{O}_F$ and $\ol{\alpha} \mc{O}_F$ are coprime, we must have $s, s - r \ge 0$. Since $\alpha$ and $\beta$ are both Weil $q$-numbers of weight one, the only possibilities are $r = s = 1$ and $s = 0$, $r = -1$, which both contradict the assumption $\beta \notin \{\alpha, \ol{\alpha}\}$.
	\end{proof}
	Throughout, we will use the notation $\otimes_{F}^g H^1(\ol{E}, \Q_\ell)$ as shorthand for 
	$\bigotimes_{F \otimes_\Q \Q_\ell}^g H^1(\ol{E}, \Q_\ell)$. One may check that
	\begin{gather*}
		\chi(\otimes_{F}^g H^1(\ol{E}, \Q_\ell), t) = (1-\alpha^g t) (1-\ol{\alpha}^g t).
	\end{gather*}
	
	\section{Decomposition by endomorphisms}
	Let $E$ be an elliptic curve over a field $k$, with $F = \End^0_k(E)$ an imaginary quadratic field. There is an induced action of $F$ on the Chow--K\"unneth summand $\mf{h}^1(E)$ of $\mf{h}(E)$, which we may then consider as a motive with coefficients in $F$. In the case where $k$ has characteristic zero, the following result has essentially already appeared in \cite[Example 2.8]{Cao2018}. We give a different proof here over an arbitrary field.
	
	\begin{lem}\label{lemH1Dual}
		The following are true:
		\begin{enumerate}[(i)]
			\item $\overline{\mf{h}^1(E)}(1)$ is the dual of $\mf{h}^1(E)$ in $\mc{M}(k)_F$;
			\item $\mf{h}^1(E) \otimes_F \overline{\mf{h}^1(E)}(1) \cong \mathbf{1}_F$.
		\end{enumerate}
	\end{lem}
	\begin{proof}
		By definition, the dual of $\mf{h}^1(E)$ in $\mc{M}(k)_F$ is $M(1)$, where $M = (E, \pi^t, 0)$ is cut out by the transpose of the idempotent $\pi$ defining $\mf{h}^1(E)$. In this case, $\pi = \Delta - [E \times_k 0] - [0 \times_k E]$ (where $\Delta$ denotes the diagonal), so $\pi^t = \pi$, and $\mf{h}^1(E)$ and $M$ have the same underlying motive in $\mc{M}(k)_\Q$.
		\par
		The $F$-action on $M$ is the transpose of the action on $\mf{h}^1(E)$: for $D$ a positive integer such that $r = \sqrt{-D} \in \End_k(E)$, $r$ acts on $\mf{h}^1(E)$ via the cycle $\psi_r = \pi \circ \Gamma_r \circ \pi \in \CH^1(E \times E)$, and on $M$ by $\psi_r^t$. We see from \cite[\hphantom{}1.10]{Scholl1994}  that $\Gamma_r^t \circ \Gamma_r$ is multiplication by the degree $D$ of the isogeny $r$ on $\mf{h}(E)$, and likewise $\psi_r^t \circ \psi_r$ is multiplication by $D$ on $\mf{h}^1(E)$. But $\psi_{-r} \circ \psi_{r} = D$ on $\mf{h}^1(E)$ also. Lemma \ref{lemMotiveCalcs} implies that $\End_{\mc{M}(k)}(\mf{h}^1(E)) \cong F$, so we must have $\psi_r^t = \psi_{-r}$. In other words, the action of $F$ on $M$ is the conjugate of the action on $\mf{h}^1(E)$, and $M \cong \overline{\mf{h}^1(E)}$.
		\par
		To see (ii), note that $\mf{h}^1(E)$, considered as an object of $\mc{M}(k)_F$, is oddly finite-dimensional in the sense of \cite[Definition 3.7]{Kimura2005}, and has rank $\Tr(\id_{\mf{h}^1(E)}) = 1$. Now \cite[Corollaire 3.19]{Andre2005} implies that 
		\begin{gather*}
			\mf{h}^1(E) \otimes_F \mf{h}^1(E)^\vee \cong \mf{h}^1(E) \otimes_F \overline{\mf{h}^1(E)}(1)
		\end{gather*}
		is isomorphic to the unit $\mathbf{1}_F$ of $\mc{M}(k)_F$.
	\end{proof}
	\begin{cor}\label{corQTensorPowerCMDecomp}
		For $g \ge 1$, there is an isomorphism in $\mc{M}(k)_F$
		\begin{equation}\label{h1gdecomp}
			\otimes_\Q^g \mf{h}^1(E) \cong \bigoplus_{\substack{i + 2j = g \\ i, j \ge 0}} a_{i, j}\cdot (\otimes_F^{i} \mf{h}^1(E))(-j), 
		\end{equation}
		where
		\begin{equation}\label{eqEgDecompClosedForm}
			a_{i, j} = \begin{cases}
				\binom{i+2j}{j} & i \ge 1 \textrm{ and } j \ge 0 \\
                \binom{2j-1}{j} & i = 0 \textrm{ and } j \ge 1 \\
                1 & (i, j) = (0, 0),
			\end{cases}
		\end{equation}
		and the $F$-action on $\otimes_\Q^g \mf{h}^1(E)$ is inherited from the first tensor factor. When $E$ is ordinary, this is the unique decomposition of $\otimes_\Q^g \mf{h}^1(E)$ into summands of the form $(\otimes_F^i \mf{h}^1(E))(-j)$, up to non-canonical isomorphism.
	\end{cor}
	\begin{proof}
		Let $m \ge 1$. Using formula \eqref{productDecompOverK} and Lemma \ref{lemH1Dual}(ii),
		\begin{align}
			\mf{h}^1(E) \otimes_\Q \otimes_F^m \mf{h}^1(E)  &\cong \otimes_\Q^2 \mf{h}^1(E) \otimes_F \otimes_F^{m-1} \mf{h}^1(E)\nonumber \\ 
			&\cong (\otimes_F^2 \mf{h}^1(E) \oplus \mathbf{1}_F(-1)) \otimes_F \otimes_F^{m-1} \mf{h}^1(E)  \nonumber \\ \label{eqCMDecompRecurrence}
			&\cong \otimes_F^{m+1} \mf{h}^1(E) \oplus (\otimes_F^{m-1} \mf{h}^1(E))(-1).
		\end{align}
        When $m = 0$,
        \begin{gather}\label{eqCMDecompRecurrenceBd}
            \mf{h}^1(E) \otimes_\Q \otimes_F^0 \mf{h}^1(E) \cong \mf{h}^1(E) \otimes_\Q \mathbf{1}_F \cong 2 \cdot \mf{h}^1(E).
        \end{gather}
		Now inducting on $g$ to write
		\begin{align*}
			\otimes_\Q^{g+1} \mf{h}^1(E) &\cong \mf{h}^1(E) \otimes_\Q \otimes_\Q^{g} \mf{h}^1(E) \\
			&\cong \mf{h}^1(E) \otimes_\Q \bigoplus_{i + 2j = g} b_{i, j}\cdot (\otimes_F^{i} \mf{h}^1(E))(-j), 
		\end{align*}
		then applying \eqref{eqCMDecompRecurrence} and \eqref{eqCMDecompRecurrenceBd} shows that the $b_{i, j}$ satisfy the recurrence relation
        \begin{equation}
			b_{i, j} = \begin{cases}
				0 & i < 0 \text{ or } j < 0 \\
				1 & i \ge 0 \text{ and } j = 0 
                \\
                2 b_{0, j} + b_{2, j-1} & i = 1 \text{ and } j > 0
                \\
				b_{i-1, j} + b_{i+1, j-1} & \text{ otherwise}.
			\end{cases}
		\end{equation}
        This uniquely determines the $b_{i, j}$, by induction on $i + 2j$.
        \par
        Let us check that the proposed solution \eqref{eqEgDecompClosedForm} satisfies the recurrence, so that $a_{i, j} = b_{i, j}$. For this purpose we will set $a_{i, j} = 0$ for $i < 0$ or $j < 0$. Clearly $a_{i, j} = 1$ for $i \ge 0$ and $j = 0$. For $i \ge 2$ and $j > 0$, we have 
        \begin{gather*}
            a_{i-1, j} + a_{i+1, j-1} = \binom{i+2j-1}{j} + \binom{i+2j-1}{j-1} = \binom{i+2j}{j} = a_{i, j},
        \end{gather*}
        by Pascal's identity. For $i = 0$ and $j > 0$, we have
        \begin{gather*}
            a_{-1, j} + a_{1, j-1} = a_{1, j-1} = \binom{2j-1}{j-1} = \binom{2j-1}{j} = a_{0, j}.
        \end{gather*}
        For $i = 1$ and $j > 0$, we have
        \begin{gather*}
            2 a_{0, j} + a_{2, j-1} = 2 \binom{2j-1}{j} + \binom{2j}{j-1}.
        \end{gather*}
        Now 
        \begin{gather*}
            2 \binom{2j-1}{j} = \frac{2j}{j} \cdot \frac{(2j-1)!}{j!(j-1)!} = \binom{2j}{j},
        \end{gather*}
        so 
        \begin{gather*}
            2 \binom{2j-1}{j} + \binom{2j}{j-1} = \binom{2j}{j} + \binom{2j}{j-1} = \binom{2j+1}{j} = a_{1, j},
        \end{gather*}
        as required. 
        \par
        The final claim about the uniqueness of the decomposition of $\otimes^g_\Q \mf{h}^1(E)$ is implied by Theorem \ref{thmUniquenessOfCMDecomp}. 
	\end{proof}
	
	In fact, we have a more general uniqueness result for this kind of decomposition, which we will later use in our computation of the transcendental motive of the cube of an ordinary CM elliptic curve (Theorem \ref{thmTransMotOfE3}). 
	
	\begin{thm}\label{thmUniquenessOfCMDecomp}
		Suppose there exists an isomorphism in $\mc{M}_{\hom}(k)_\Q$
		\begin{gather}\label{eqCMDecomp1}
			\Psi: \bigoplus\limits_{1 \le i \le n} \left[\otimes_{F_i}^{r_i} \mf{h}^1(C_i)\right](s_i)_{\hom} \cong \bigoplus\limits_{1 \le j \le m} \left[\otimes_{L_j}^{t_j} \mf{h}^1(D_j)\right](u_j)_{\hom},
		\end{gather}
		where each $C_i$ (resp. $D_j$) is an ordinary elliptic curve over $k$ with CM by $F_i$ (resp. $L_j$), and $r_i,\ s_i,\ t_j,\ u_j \in \Z$ with $r_i,\ t_j \ge 0$. Then $m = n$, and there exists a permutation $\sigma$ of $\{1,...,n\}$ such that for all $i$, $r_i = t_{\sigma(i)}$, $s_i = u_{\sigma(i)}$, and if $r_i \neq 0$, $F_i \cong L_{\sigma(i)}$.
	\end{thm}
	
	The idea of the proof of Theorem \ref{thmUniquenessOfCMDecomp} is to reduce to the case where $k$ is a finite field, then apply Lemma \ref{lemFrobEigenvalue}. The reduction step requires the following standard lemmas. 
	
	\begin{lem}\label{lemRedGalRepIso}
		Let $R$ be a discrete valuation ring, $K$ its fraction field and $\kappa$ its residue field, with $\character(\kappa) \neq \ell$. Let $E$ be an ordinary elliptic curve over $K$ with CM by $F$, having good, ordinary reduction $E_\kappa$. Then $E_\kappa$ also has CM by $F$, and there is a canonical $F \otimes_\Q \Q_\ell$-linear isomorphism of $\Gal(\ol{\kappa} / \kappa)$ representations 
		\begin{gather*}
			H^1(E_{\ol{K}}, \Q_\ell) \cong H^1(E_{\ol{\kappa}}, \Q_\ell).
		\end{gather*}
	\end{lem}
	\begin{proof}
		This follows from the universal property of the N\'eron model, which gives a map $\End^0_K(E) \to \End^0_\kappa(E_\kappa)$ compatible with the isomorphism of \cite[Lemma 2]{ST1966}. 
	\end{proof}

    \begin{lem}\label{lemSimGoodRed}
        If $E_1,...,E_n$ are CM elliptic curves over a number field $L$, there exist infinitely many primes of $L$ where all of the $E_i$ simultaneously have good, ordinary reduction.
    \end{lem}
    \begin{proof}
        By \cite[Theorem 13.12]{Lang1987}, if $\mf{p} \in \Spec \mc{O}_L$ is a prime of good reduction for $E_i$, then $E_i$ has ordinary reduction at $\mf{p}$ iff $\mf{p}$ lies a over a rational prime which splits in the CM-field of $E_i$. The claim thus follows from the Chebotarev density theorem.
    \end{proof}
	
	\begin{proof}[Proof of Theorem \ref{thmUniquenessOfCMDecomp}]
		Taking $\ell$-adic cohomology on both sides of the isomorphism \eqref{eqCMDecomp1}, then comparing dimensions shows that $m = n$.
		\par
		Extending $k$ if necessary, we may assume without loss of generality that all $C_i$ and $D_j$, and their respective geometric endomorphism algebras, are defined over a finite extension $M \subseteq k$ of the prime field of $k$. Via a standard spreading out argument, we may further assume that $k$ is finitely generated, and that $M$ is algebraically closed in $k$. Taking the $\ell$-adic realization of \eqref{eqCMDecomp1}, we obtain 
		\begin{gather}\label{eqSpecialIsoOverFGField}
			\bigoplus\limits_{1 \le i \le n} \left[\otimes_{F_i}^{r_i} H^1((C_i)_{\ol{k}}, \Q_\ell)\right](s_i) \cong \bigoplus\limits_{1 \le j \le n} \left[\otimes_{L_j}^{t_j} H^1((D_j)_{\ol{k}}, \Q_\ell)\right](u_j).
		\end{gather}
		Since all $C_i$, $D_j$ are defined over $M$, the $G_k$-action in \eqref{eqSpecialIsoOverFGField} factors through the surjective restriction $G_k \to G_M$, and we thus reduce to the case where $k$ is finite over the prime field.
		\par
		Suppose that $k$ is a number field. By Lemma \ref{lemSimGoodRed}, there exists a prime $\mf{p} \in \Spec \mc{O}_k$ where all $C_i$ and $D_j$ simultaneously have good, ordinary reduction, such that $\ell$ is invertible in the residue field $\kappa_{\mf{p}}$. Restricting the representations appearing in \eqref{eqSpecialIsoOverFGField} to the decomposition group $G_\mf{p}$ of a prime of $\ol{k}$ over $\mf{p}$, and applying Lemma \ref{lemRedGalRepIso}, allows us to replace $k$ with a finite field in \eqref{eqSpecialIsoOverFGField}.
		\par
		We may thus assume that $k \cong \F_q$ is finite. Consider the determinant of the operator $1 - \Frob_q t$ acting on each side of \eqref{eqSpecialIsoOverFGField}. By the Weil conjectures, this gives an equality of polynomials
		\begin{gather*}
			\prod\limits_{1 \le i \le n} (1-q^{-s_i} \alpha_i^{r_i} t)(1-q^{-s_i} \ol{\alpha_i}^{r_i} t) = \prod\limits_{1 \le j \le n} (1-q^{-u_j} \beta_j^{t_j} t)(1-q^{-u_j} \ol{\beta_j}^{t_j} t),
		\end{gather*}
		where $\alpha_i \in F_i$, $\beta_j \in L_j$ are algebraic integers. It follows that up to reordering the $\beta_j$, for all $i$,
		\begin{gather}\label{eqFrobCharPolysOverFiniteField}
			(1-q^{-s_i} \alpha_i^{r_i} t)(1-q^{-s_i} \ol{\alpha_i}^{r_i} t) = (1-q^{-u_i} \beta_i^{t_i} t)(1-q^{-u_i} \ol{\beta_i}^{t_i} t).
		\end{gather}
		If $r_i = t_i = 0$, then $s_i = u_i$. Otherwise, by Lemma \ref{lemFrobEigenvalue}, the above polynomial is irreducible over $\Q$, and $F_i \cong L_i$ is its splitting field. To complete the proof it suffices to show that $r_i = t_i$ and $s_i = u_i$ in this latter case.
		\par
		Comparing like coefficients in equation \eqref{eqFrobCharPolysOverFiniteField}, we see 
		\begin{align}\label{eqWeights}
			q^{r_i - 2s_i} &= q^{t_i  -2u_i}, \\
			q^{-s_i} \Tr(\alpha_i^{r_i}) &= q^{-u_i} \Tr(\beta_i^{t_i}).
		\end{align}
		Combining these two equations yields
		\begin{gather}\label{eqFrobCharPolyTrace}
			\Tr(\alpha_i^{r_i}) = q^{\frac{1}{2}(r_i-t_i)} \Tr(\beta_i^{t_i}).
		\end{gather}
		Recall that $r_i$ and $t_i$ are positive. By Lemma \ref{lemFrobEigenvalue}, both $\Tr(\alpha_i^{r_i})$ and $\Tr(\beta_i^{t_i})$ have valuation zero at all primes of $\mc{O}_{F_i}$ over $p$,  so equation \eqref{eqFrobCharPolyTrace} implies $r_i = t_i$. Now equation \eqref{eqWeights} shows $s_i = u_i$. 
	\end{proof}
	
	We now prove some cases of the Tate conjecture.
	\begin{lem}\label{lemNoInvariants}
		Let $E$ be an ordinary elliptic curve over a finitely generated field $k$ with CM by $F$. Then
		\begin{enumerate}[(i)]
			\item for $r \ge 1$ and $s \in \Z$,
			\begin{gather*}
				H^\bullet(\otimes^r_F \mf{h}^1(E), \Q_\ell(s))^{G_k} = 0;
			\end{gather*}
			\item For $X$ a smooth projective variety over $k$ of dimension $d$, $r > 0$, $t \ge 0$, and $s \in \Z$, such that either $r > t$ or $r > 2d - t$, 
			\begin{gather*}
				[H^\bullet(\otimes^r_F \mf{h}^1(E), \Q_\ell(s)) \otimes_{\Q_\ell} H^t(\ol{X}, \Q_\ell)]^{G_k} = 0.
			\end{gather*}
		\end{enumerate}
	\end{lem}
	\begin{proof}
		Claim (i) is implied by the $t = 0$ case of claim (ii), so it suffices to prove (ii). By passing to a finite extension we can assume $E$ is constant over $k$, and thus that $k$ has a finite place $v$ where $E$ and $X$ both have good reduction. Using Lemma \ref{lemSimGoodRed}, we may further assume that $E$ has ordinary reduction at $v$. Then by examining the characteristic polynomial of a Frobenius element at $v$ (and applying smooth proper base change), the claim reduces to the case where $k$ is finite. The case where $r > t$ then follows from Lemma \ref{lemFrobEigenvalue} and the Weil conjectures. This implies the case $r > 2d - t$ by Poincar\'e duality on $X$.
	\end{proof}
	
	\begin{cor}\label{corAllCyclesHomTrivial}
		If $k$ is an arbitrary field, and $E$ is an ordinary elliptic curve over $k$ with CM by $F$, then for all $g \ge 1$ and $i \in \Z$, $\CH^i(\otimes_F^g \mf{h}^1(E))$ consists of homologically trivial cycles. If $X$ is a smooth projective variety over $k$, then for all $g \ge 2$ and $i \in \Z$, $\CH^i(\otimes_F^g \mf{h}^1(E) \otimes_\Q \mf{h}^1(X))$ consists of homologically trivial cycles.
	\end{cor}
    \begin{proof}
        This reduces to the case of a finitely generated field, where it follows from Lemma \ref{lemNoInvariants}.
    \end{proof}
	
	\begin{cor}\label{corTateForEg}
		If $k$ is a finitely generated field, $E$ a CM elliptic curve over $k$, $C$ a smooth projective curve over $k$, and $g \ge 0$, the full Tate conjecture holds for $E^g \times_k C$.
	\end{cor}
	\begin{proof}
		For now assume $E$ is ordinary and $C$ is geometrically integral. The motive of $C$ decomposes as
		\begin{gather*}
			\mf{h}(C) \cong \mathbf{1} \oplus \mf{h}^1(C) \oplus \mathbf{1}(-1).
		\end{gather*}
		When we decompose $\mf{h}(E^g)$ and apply Lemma \ref{lemNoInvariants}, we see that $\mf{h}(E^g \times_k C)$ is a direct sum of motives which either have no Tate classes in their cohomology, are Tate twists of the unit motive, or are Tate twists of $\mf{h}^1(E) \otimes_\Q \mf{h}^1(C)$. It thus suffices to prove the Tate conjecture for motives of the latter form, which follows from the Tate conjecture for divisors on $E \times_k \Jac(C)$ (\cite{Tate1966}, \cite{Faltings1986}, \cite{Zarhin1974a}, \cite{Zarhin1974b}; see also Mori's results in \cite[Chapitre VI.5]{MB1985}).
		\par
		Now suppose $E$ is supersingular. After passing to a finite extension, $\mf{h}(E^g)$ decomposes into a sum of Tate twists of $\mathbf{1}$ and $\mf{h}^1(E)$ by \cite[Section 2]{Moonen2024a}, and the same argument as above works.
		\par
		The extension to the case where $C$ is not geometrically integral is straightforward and omitted.
	\end{proof}
	
	\begin{rem}
		The full Tate conjecture for powers of CM elliptic curves was proved by Tate over number fields \cite[page 106]{Tate1965}, and since that time has been well-known over arbitrary finitely-generated fields. 
	\end{rem}
	
	\par
	Let $C$ be a smooth, projective, geometrically integral curve over $\F_q$, and $K$ its function field.
	\begin{cor}\label{corBBForEg}
		For $E$ an elliptic curve over $K$ which acquires CM over $\ol{K}$, and $g, i > 0$,
		$BB^i(E^g)$ holds. 
	\end{cor}
	\begin{proof}
		There exists a finite extension $L / K$ such that $E_L$ is isogenous to a constant elliptic curve over $L$. Thus, Corollaries \ref{corBBForAbMotsBaseChange} and \ref{corTateForEg} yield the result.
	\end{proof}
    \begin{rem}
        Recall that via Remark \ref{remJannsenContrastDetailed}, we can also control all the other motivic cohomology groups of a curve $E$ as in Corollary \ref{corBBForEg}.
    \end{rem}
	We now revisit Moonen's question mentioned in the introduction, and show that conditional on some strong conjectures, it has a positive answer in positive characteristic.
	
	\begin{lem}\label{lemConcentrationOfCH}
		Let $E$ be an ordinary elliptic curve and $X$ a smooth projective variety over $\F_q$. If $g > 0$, $g \neq i$ and the $\ell$-adic cycle class map is injective on ${\CH^i(\mf{h}(X) \otimes_\Q \otimes_F^g \mf{h}^1(E))}$, then $\CH^i(\otimes_F^g\mf{h}^1(E_{k(X)})) = 0$, where $k(X)$ is the function field of $X$. 
	\end{lem}
	\begin{proof}
		For $i > g$ the claim is trivial, as $\otimes_F^g \mf{h}^1(E_{k(X)})$ is a summand of $\mf{h}(E^g_{k(X)})$. For $i < g$, $H^{2i}(\mf{h}(X) \otimes_\Q \otimes_F^g \mf{h}^1(E), \Q_\ell(i))$ has no Galois invariants by Lemma \ref{lemNoInvariants}. By the injectivity hypothesis, we find
		\begin{gather*}
			\CH^i(\mf{h}(X) \otimes_\Q \otimes_F^g \mf{h}^1(E)) = 0.
		\end{gather*}
		But this group surjects onto $\CH^i(\otimes_F^g \mf{h}^1(E_{k(X)}))$, as the localization exact sequence \eqref{eqLocalization} is compatible with the relevant correspondences.
	\end{proof}
	
	\begin{thm}\label{thmChVanishOverFuncField}
		For an ordinary CM elliptic curve $E$ over a global function field $K$, $g \ge 2$, and $i \in \Z$,
		\begin{gather*}
			\CH^i(\otimes_F^g \mf{h}^1(E)) = 0.
		\end{gather*}
	\end{thm}
    \begin{proof}
        It suffices to assume that $E$ is defined over $\F_q$. Let $C$ be the smooth projective curve with function field $K$. Using Lemma \ref{lemTateEquivs} and Corollary \ref{corTateForEg} to check the needed injectivity of the cycle class map on $\CH^\bullet(\mf{h}(C) \otimes_\Q \otimes_F^g \mf{h}^1(E))$, the claim follows from essentially the same argument as in Lemma \ref{lemConcentrationOfCH}.
    \end{proof}
	
	\begin{cor}\label{corGeneralConcentrationOfCH}
		Suppose that for all smooth projective varieties $X$ over $\F_p$, the $\ell$-adic cycle class map is injective on $\CH^i(X)$. Then for any field $k$ of positive characteristic $p$, and any ordinary elliptic curve $E$ over $k$ with CM by $F$, $\CH^i(\otimes_F^g \mf{h}^1(E)) = 0$ for $i \neq g$.
	\end{cor}
	\begin{proof}
		It suffices to assume that $k$ is finitely generated. We can then without loss of generality pass to a finite extension of $k$ and assume that $E$ is defined over a finite subfield of $k$. Finally, by applying an alteration \cite[Theorem 4.1]{dJ1996}, we reduce to the case where $k$ is the function field of a smooth projective variety over $\F_p$. Lemma \ref{lemConcentrationOfCH} now implies the claim.
	\end{proof}
	
	\begin{cor}\label{corAlgEqualsNumForEg}
		Let $E$ be a CM elliptic curve over a field $k$ of positive characteristic $p$. Under the same hypothesis as in Corollary \ref{corGeneralConcentrationOfCH}, for all $g > 0$, $\CH^\bullet_{\alg}(E^g) = \CH^\bullet_{\num}(E^g)$.
	\end{cor}
	\begin{proof}
		It suffices to assume that $k$ is algebraically closed. The case where $E$ is supersingular does not depend on any conjectures, and is \cite[corollary 2.12(iii)]{LF2021}.
		\par
		Suppose $E$ is ordinary, with endomorphism algebra $F$. By induction on $g$ it suffices to show $\CH^\bullet_{\alg}(\otimes_F^g \mf{h}^1(E)) = \CH^\bullet_{\num}(\otimes_F^g \mf{h}^1(E))$. But by Corollary \ref{corGeneralConcentrationOfCH}, $\CH^\bullet(\otimes_F^g \mf{h}^1(E)) \subseteq \CH^g(E^g)$, and algebraic and numerical equivalence agree for zero-cycles on $E^g$ \cite[Proposition 10.2]{Moonen2024b}.
	\end{proof}
	
	\appendix
	\section{The transcendental motives of $E^2$ and $E^3$}
	The transcendental motive of a surface or a nice\footnote{See \cite[Hypothesis 7.1]{Kahn2021} for the meaning of \enquote{nice}.} threefold $X$ is a summand of $\mf{h}(X)$ whose isomorphism class is canonically defined, and which in some sense contains the most subtle part of the Chow groups and cohomology of $X$. The results we have developed thus far allow us to compute the transcendental motives of $E^2$ and $E^3$, where $E$ is an ordinary CM elliptic curve over an arbitrary field $k$. 
	\par
	We will begin by reviewing the definitions and properties of transcendental motives. We caution that our references for transcendental motives, \cite{Kahn_Murre_Pedrini_2007} and \cite{Kahn2021}, use different conventions from ours. Their results are phrased in terms of covariant motives, meaning that they work in the opposite category of our $\mc{M}_\sim(k)_L$. They also use the opposite sign convention for Tate twists. We will translate to our preferred conventions when stating their results.
	
	\subsection{Transcendental motives of surfaces}
	Let $S$ be a smooth projective surface over $k$, and let $T(S)$ denote the Albanese kernel of $S$, i.e.\ the kernel of the natural map from the group $\CH^2(S)_{0, \Z}$ of zero-cycles of degree zero on $S$ to the $k$-points of the Albanese variety $\Alb(S)$.
	\begin{thm}[{{\cite[Proposition 14.2.3]{Kahn_Murre_Pedrini_2007}}}]\label{thmTransMotOfSurface}
		Let $\mf{h}(S) \cong \bigoplus\limits_{i=0}^4 \mf{h}^i(S)$ be a Chow--K\"unneth decomposition of the form constructed in \cite[Section 4]{Scholl1994}. Let 
		\begin{equation*}
			\ul{\NS}_S = \NS(S_{k_s}) \otimes \Q
		\end{equation*}
		be the rational, geometric N\'eron-Severi group of $S$, viewed as a $G_k$-module. There exists a unique decomposition 
		\begin{gather*}
			\mf{h}^2(S) \cong \mf{h}^2_{\alg}(S) \oplus t^2(S)
		\end{gather*}
		such that $\mf{h}^2_{\alg}(S) \cong \mf{h}(\ul{\NS}_S)(-1)$. Here $\mf{h}(\ul{\NS}_S)$ is the Artin motive associated to $\ul{\NS}_S$. Moreover, the Chow groups of $t^2(S)$ satisfy $\CH^0(t^2(S)) = \CH^1(t^2(S)) = 0$, $\CH^2(t^2(S)) = T(S) \otimes_\Z \Q$.
	\end{thm}
	
	The motive $t^2(S)$ is called the transcendental motive of $S$. When $S$ is an abelian surface, the canonical Chow--K\"unneth decomposition satisfies the hypothesis of Theorem \ref{thmTransMotOfSurface} (by \cite[Theorem 5.3]{Scholl1994}).
	
	\subsection{The coniveau decomposition}
	Jannsen in \cite[Theorem 1]{Jannsen1992} showed that the category of numerical motives $\mc{M}_{\num}(k)_L$ with coefficients in a characteristic zero field $L$ is a semisimple abelian category. Call a motive \textbf{effective} if it is isomorphic to a summand of the motive of a variety, and let $\mc{M}_{\num}^{\eff}(k)_L$ denote the full subcategory of $\mc{M}_{\num}(k)_L$ on the effective motives.
	
	\begin{defn}[{{\cite[Definition 4.1]{Kahn2021}}}]
		A simple motive $S \in \mc{M}_{\num}^{\eff}(k)_L$ is called \textbf{primitive} if $S(1)$ is not effective. 
	\end{defn}
	
	\begin{defn}[{{\cite[Lemma 4.2]{Kahn2021}}}]
		The \textbf{coniveau} of a simple motive $S \in \mc{M}_{\num}^{\eff}(k)_L$ is the unique $n \ge 0$ such that $S(n)$ is primitive.
	\end{defn}
	
	A direct sum of simple motives of coniveau $n$ is also defined to have coniveau $n$. 
	By semisimplicity, every $M \in \mc{M}_{\num}^{\eff}(k)_L$ has a unique decomposition
	\begin{gather*}
		M \cong \bigoplus_{j \ge 0} M_j(-j),
	\end{gather*}
	with $M_j$ primitive. 
	\par
	The next lemma is a special case of \cite[Lemma 7.5 and Theorem 7.7]{Kahn2021}. 
	\begin{lem}
		\label{lemConiveauLiftAbelian3Fold}
		
		For an abelian threefold $A$ over $k$, the coniveau decomposition of $\mf{h}^3(A)_{\num}$ lifts to rational equivalence, uniquely up to non-canonical isomorphism. The lift takes the form
		\begin{gather*}
			\mf{h}^3(A) \cong t^3(A) \oplus \mf{h}^1(J)(-1)
		\end{gather*}
		for a certain abelian variety $J$, such that $t^3(A)_{\num}$ has coniveau zero, and
		\begin{gather*}
			\CH^2_{\alg}(t^3(A)) \cong \Griff(A),
		\end{gather*}
        where $\Griff(A)$ is the numerical Griffiths group.
	\end{lem}
	
	We call $t^3(A)$ the transcendental motive of $A$. 
	
	\subsection{The transcendental motive of $E^2$}
	To compute $t^2(E^2)$ we will need to know the N\'eron-Severi group of $E^2$, which is determined by the following classical result. For completeness we sketch a proof here.
	
	\begin{thm}\label{thmNSE2}
		Let $X$ and $Y$ be smooth, projective, geometrically integral varieties over a field $k$, $\sigma_X \in X(k)$, and $\sigma_Y \in Y(k)$. The N\'eron-Severi group of $X \times_k Y$ is given by
		\begin{gather*}
			\NS(X \times_k Y) \cong \NS(X) \oplus \NS(Y) \oplus 
			\Hom_{\GrpSch_k}(\Alb(X), \Pic^0(Y)).
		\end{gather*}
	\end{thm}
	\begin{proof}
		From \cite[Corollary 4.4.8]{vDdB2018}, we get a certain isomorphism of sheaves on the fppf site of a scheme $S$. Taking $X_1 = X$, $X_2 = Y$, $S = \Spec k$, $\sigma_1 = \sigma_X$ and $\sigma_2 = \sigma_Y$ in the corollary, then passing to global sections yields an isomorphism of abelian groups
		\begin{gather*}
			\Pic(X \times_k Y)(k) \cong \Pic(X)(k) \oplus \Pic(Y)(k) \oplus 
			\Hom_{\GrpSch_k}(\Alb(X), \Pic^0(Y)),
		\end{gather*}
		By Corollary 4.4.9 in loc. cit., the isomorphism maps $\Pic^0(X \times_k Y)(k)$ onto $\Pic^0(X)(k) \oplus \Pic^0(Y)(k)$, so 
		\begin{gather*}
			\NS(X \times_k Y) \cong \NS(X) \oplus \NS(Y) \oplus 
			\Hom_{\GrpSch_k}(\Alb(X), \Pic^0(Y)).
		\end{gather*}
	\end{proof}
	
	Let $E$ be an ordinary CM elliptic curve over $k$ with endomorphism algebra $F$.
	
	\begin{thm}\label{thmTransMotiveOfE2}
		The transcendental motive $t^2(E^2)$ of $E^2$ is isomorphic in $\mc{M}(k)_\Q$ to $\otimes_F^2 \mf{h}^1(E)$.
	\end{thm}
	\begin{proof}
		Corollary \ref{corQTensorPowerCMDecomp} decomposes $\mf{h}^2(E^2)$ as 
		\begin{gather}\label{eqOCMDecompH2E2}
			\mf{h}^2(E^2) \cong \otimes_F^2 \mf{h}^1(E) \oplus 4 \cdot \mathbf{1}(-1),
		\end{gather}
        via the K\"unneth formula. The Artin motive of $\ul{\NS}_{E^2}$ is $4 \cdot \mathbf{1}$ by Theorem \ref{thmNSE2}, so \eqref{eqOCMDecompH2E2} is the same as the decomposition of $\mf{h}^2(E^2)$ given by Theorem \ref{thmTransMotOfSurface}.
	\end{proof}
	\begin{thm}\label{thmTransMotOfE3}
		The transcendental motive $t^3(E^3)$ of $E^3$ is isomorphic to $\otimes_F^3 \mf{h}^1(E)$.
	\end{thm}
	
	\begin{proof}
		By Lemma \ref{lemConiveauLiftAbelian3Fold},
		\begin{equation}\label{eqTransDecompH3E3}
			\mf{h}^3(E^3) \cong t^3(E^3) \oplus \mf{h}^1(J)(-1)
		\end{equation}
		for a certain abelian variety $J$ over $k$, such that $\mf{h}^1(J)(-1)_{\num}$ is the maximal coniveau one summand of $\mf{h}^3(E^3)_{\num}$. 
		\par
		Note that $\mf{h}^1(E)_{\num}$ is simple, by Lemma \ref{lemMotiveCalcs}, and primitive, by \cite[Lemma 4.4]{Kahn2021}. Corollary \ref{corQTensorPowerCMDecomp} and the K\"unneth formula show that
		\begin{gather}\label{eqh3E3Decomp}
			\mf{h}^3(E^3) \cong \otimes_F^3 \mf{h}^1(E) \oplus 9 \cdot \mf{h}^1(E)(-1),
		\end{gather} 
		so $9 \cdot \mf{h}^1(E)_{\num}$ is a summand of $\mf{h}^1(J)_{\num}$.
		\par
		There is a split epimorphism $\mf{h}^1(J)_{\num} \to 9 \cdot \mf{h}^1(E)_{\num}$, so by Lemma \ref{lemMotiveCalcs}, $J$ is isogenous to a product $E^9 \times B$, where $B$ is some abelian variety. By the K\"unneth formula, $\dim_{\Q_\ell} H^3(\mf{h}^3(E^3), \Q_\ell) = 20$, so that on applying $H^3(\bullet, \Q_\ell)$ to both sides of \eqref{eqTransDecompH3E3}, we see that either $B = 0$, or $B$ is one-dimensional.
		\par
		In the case $B = 0$, we have $\mf{h}^1(J) = 9 \cdot \mf{h}^1(E)$. But then by semisimplicity of numerical motives, both decompositions \eqref{eqTransDecompH3E3} and \eqref{eqh3E3Decomp} induce the coniveau decomposition of $\mf{h}^3(E^3)_{\num}$
		upon passing to numerical equivalence. By uniqueness of the lift of the coniveau decomposition to rational equivalence, $t^3(E^3) \cong \otimes_F^3 \mf{h}^1(E)$.
		\par
		To finish the proof it will suffice to show that $B$ cannot be one-dimensional. Suppose otherwise. Then
		\begin{gather*}
			\dim_{\Q_\ell} H^3(\mf{h}^1(J)(-1), \Q_\ell) = \dim_{\Q_\ell} H^3(E^3, \Q_\ell),
		\end{gather*}
		so $H^\bullet(t^3(E^3), \Q_\ell) = 0$, and by conservativity (Corollary \ref{realDetectsZero}), $t^3(E^3) = 0$. Thus 
		$\mf{h}^3(E^3) \cong \mf{h}^1(J)(-1),$
		and 
		\begin{gather}\label{eqContrIsoOfOCMMotives}
			\otimes_F^3 \mf{h}^1(E) \oplus 9 \cdot \mf{h}^1(E)(-1) \cong \mf{h}^1(B)(-1) \oplus 9 \cdot \mf{h}^1(E)(-1).
		\end{gather}
		By semisimplicity of numerical motives, $\otimes_F^3 \mf{h}^1(E)_{\num} \cong \mf{h}^1(B)(-1)_{\num}$. The endomorphism algebra of $\otimes_F^3 \mf{h}^1(E)_{\num}$ is isomorphic to $F$, so by Lemma \ref{lemMotiveCalcs}, $B$ has ordinary CM. But then \eqref{eqContrIsoOfOCMMotives} contradicts Theorem \ref{thmUniquenessOfCMDecomp}.
	\end{proof}
	
	\begin{cor}\label{corCH2t3OverFuncField}
		If $K$ is the function field of a smooth projective variety $X$ over $k$, then
		\begin{gather*}
			\CH^2(\otimes_F^3 \mf{h}^1(E_K)) \cong \CH^2(\mf{h}(X) \otimes_\Q \otimes_F^3 \mf{h}^1(E)).
		\end{gather*}
	\end{cor}
	\begin{proof}
		Theorem \ref{thmTransMotOfE3} and \cite[Proposition 7.6]{Kahn2021} imply that
		\begin{gather*}
			\CH^2(\otimes_F^3 \mf{h}^1(E_K)) \cong \Hom_{\mc{M}(k)_\Q}(\otimes_F^3 \mf{h}^1(E), \mf{h}(X)(-1)).
		\end{gather*}
		The dual of $\otimes_F^3 \mf{h}^1(E)$ in $\mc{M}(k)_\Q$ is $\otimes_F^3 \mf{h}^1(E)(3)$, so 
		\begin{align*}
			\CH^2(\otimes_F^3 \mf{h}^1(E_K)) &= \Hom_{\mc{M}(k)_\Q}(\otimes_F^3 \mf{h}^1(E), \mf{h}(X)(-1)) \\
            &= \Hom_{\mc{M}(k)_\Q}(\mathbf{1}, \mf{h}(X) \otimes_\Q \otimes_F^3 \mf{h}^1(E)(2)) \\
            &= \Hom_{\mc{M}(k)_\Q}(\mathbf{1}(-2), \mf{h}(X) \otimes_\Q \otimes_F^3 \mf{h}^1(E)) \\
			&= \CH^2(\mf{h}(X) \otimes_\Q \otimes_F^3 \mf{h}^1(E)).
		\end{align*}
        Here we have used that $\Hom(A, B) = \Hom(\mathbf{1}, A^\vee \otimes B)$ in a rigid category.
	\end{proof}
	
	\printbibliography
\end{document}
\typeout{get arXiv to do 4 passes: Label(s) may have changed. Rerun}